\documentclass[twoside]{article}
\usepackage[normalem]{ulem}
\usepackage{soul}

\usepackage{amsmath,amsfonts,amssymb,amscd,amsthm,amsbsy}
\usepackage{fancyhdr,latexsym,epsfig,enumerate}
\usepackage{xcolor}
\usepackage[latin1]{inputenc}
\usepackage[T1]{fontenc}
\usepackage[all]{xy}
\usepackage{endnotes}
\usepackage{url}

\usepackage{graphicx}

\usepackage{tikz}
\usepackage{tikz-cd}
\usetikzlibrary{arrows,decorations.markings}
\tikzset{commutative diagrams/.cd,arrow style=tikz,diagrams={>=stealth'}}

\newtheorem{thm}{Theorem}[section]
\newtheorem{lem}[thm]{Lemma}
\newtheorem{prop}[thm]{Proposition}
\newtheorem{cor}[thm]{Corollary}

\theoremstyle{definition}

\newcommand{\surf}[2]{\widehat{\Sigma}_{{#1}{#2}}}
\newcommand{\bsurf}[3]{\mathop{B_{#1}(\widehat{\Sigma}_{{#2}{#3}})}}

\newcommand{\brak}[1]{\ensuremath{\left\{ #1 \right\}}}
\newcommand{\ang}[1]{\ensuremath{\langle #1\rangle}}
\newcommand{\sset}[2]{\ensuremath{\brak{#1 \, , \ldots, \, #2}}}

\newcommand{\N}{\ensuremath{\mathbb N}}
\newcommand{\Z}{\ensuremath{\mathbb Z}}
\newcommand{\D}{\ensuremath{\mathbb D}}

\newcommand{\FF}{\ensuremath{\mathbb F}}
\newcommand{\F}[1][n]{\ensuremath{\F\FF_{{#1}}}}

\renewcommand{\epsilon}{\ensuremath{\varepsilon}}
\renewcommand{\phi}{\ensuremath{\varphi}}
\renewcommand{\to}{\ensuremath{\longrightarrow}}

\newcommand{\MB}[3][,n]{\ensuremath{B_{{#2}{#1}}({#3})}}

\renewcommand{\Im}[1]{\operatorname{Im}(#1)}

\def\tic{\widetilde{c}}
\def\tia{\widetilde{a}}
\def\tib{\widetilde{b}}
\def\z{\zeta}
\def\si{\sigma}
\def\tsi{\widetilde{\sigma}}

\newlength{\wideitemsep}
\let\olditem\item
\renewcommand{\item}{\setlength{\itemsep}{\wideitemsep}\olditem}

\begin{document}

\pagestyle{myheadings}

\markboth{Paolo Bellingeri, Eddy Godelle and John Guaschi}{Abelian and metabelian quotients of surface braid groups}

\title{Abelian and metabelian quotients of surface braid groups}

\author{Paolo Bellingeri, Eddy Godelle and John Guaschi}

\date{}

\maketitle

\begin{abstract} 
In this paper we study Abelian and metabelian
quotients of braid groups fn oriented surfaces with boundary components.       
We provide group presentations and we prove rigidity results for these quotients arising from exact sequences related to (generalised) Fadell-Neuwirth fibrations.
\end{abstract}

\begingroup
\renewcommand{\thefootnote}{}
\endgroup 

\maketitle


\section{Introduction}


For $n\in \N$, let $B_{n}$ denote the Artin braid group on $n$ strings. The \emph{Burau representation}, which is known not to be faithful if $n\geq 5$, and the \emph{Bigelow-Krammer-Lawrence representation}, which is faithful~\cite{Big0,Kra}, play an important r\^{o}le in the theory of $B_{n}$~\cite{Bir,KT}. They arise as representatives of a larger family of representations of $B_n$ defined via the action on a regular covering space of the  $k\textsuperscript{th}$ permuted configuration space of the $n$-punctured disc $\D_n$ \cite{Z}. As we will explain in Section \ref{section:representations}, these regular coverings are obtained from a pair $(G_k, p_k)$, $k\ge 1$, where $p_k: B_k(\D_n) \to G_k$ is the surjective homomorphism from the $k$-string braid group of $\D_{n}$ onto a group $G_{k}$ that is defined by the following commutative diagram of short exact sequences:
\begin{equation}\label{eq:diagartin}
\begin{xy}*!C\xybox{%
\xymatrix{
1 \ar[r]  &  B_k(\D_n)  \ar[d]^{p_k}   \ar[r]    & B_{k,n}  \ar[d]^{r_{k,n}} \ar[r]          &  B_n   \ar[d]^{r_{n}} \ar[r]        & 1\\
1 \ar[r]  &  G_k  \ar[r]                          & B_{k,n}/ \Gamma_2(B_{k,n})  \ar[r] &  B_n/ \Gamma_2(B_n)  \ar[r]   & 1,}}
\end{xy}
\end{equation}
where $B_{k,n}$ is the mixed Artin braid group on $(k,n)$ strings and $\Gamma_{2}(B_{k,n})$ is the second term in the lower central series of  $B_{k,n}$ (recall that the \emph{lower central series} of $G$ is the filtration $G =\Gamma_1(G) \supseteq \Gamma_2(G) \supseteq \cdots$, 
where $\Gamma_i(G)=[G,\Gamma_{i-1}(G)]$ for $i\geq 2$), and the vertical maps $r_{k,n}$ and $r_{n}$ are the Abelianisation homomorphisms (see Section~\ref{section2} for precise definitions and Section~\ref{section4} for more details about this construction). Although similar constructions have been carried out for other related groups, such as Artin-Tits groups of spherical type~\cite{CW,D}, their generalisation in a more topological direction, to braid and mapping class groups of surfaces for example, remains a largely open problem, with the exception of a few results~\cite{Bar2,BiB,GG7}.

Let $\surf{g}{}$ be an orientable, compact surface of positive genus~$g$ and with one boundary component, and let $\surf{g}{, n}=\surf{g}{} \setminus \sset{x_{1}}{x_{n}}$, where $\sset{x_{1}}{x_{n}}$ is an $n$-point subset in the interior of $\surf{g}{}$. In \cite{HK}, An and Ko described an extension of the Bigelow-Krammer-Lawrence representations of $B_{n}$ to surface braid groups~$B_n(\surf{g}{})$ based on the regular covering associated to a projection map~$\varPhi_\Sigma: B_k(\surf{g}{,n})\to G_\Sigma$ of the braid group~$B_k(\surf{g}{,n})$ onto a specific group~$G_\Sigma$, which can be seen as a kind of generalised Heisenberg group
and  that is constructed as a subgroup of a group $H_\Sigma$. The group $H_\Sigma$ is defined abstractly in terms of its group presentation, and is chosen to satisfy certain technical homological constraints  (Section $3.1$ of \cite{HK}). An and Ko prove a rigidity result for $H_\Sigma$, which states intuitively that it is the `best possible' group that satisfies the constraints. 
However the choices of $H_\Sigma$ and $G_\Sigma$ seem to be based on \emph{ad hoc} technical arguments. 

Our first objective is to show that $G_\Sigma$ defined in  \cite{HK} may be constructed using short exact sequences of surface braid groups emanating from Fadell-Neuwirth fibrations, in which the lower central series $(\Gamma_{i})_{i\in \N}$ of $G_{\Sigma}$ plays a prominent r\^{o}le. These sequences are similar to those for $G_{k}$ given by equation~(\ref{eq:diagartin}), but as we shall see in Lemma~\ref{nolinearext}, there is a marked difference with the case of the Artin braid groups, since the short exact sequence on the $\Gamma_{2}$-level does not yield the expected group $G_\Sigma$ and homomorphism $\Phi_\Sigma$. However, we prove that at the following stage, at the $\Gamma_{3}$-level, the construction does indeed give rise to $G_\Sigma$. Consider the following commutative diagram of short exact sequences (the first line is the short exact sequence~(\ref{eq:sequence}) that we shall recall in Section~\ref{section2}):
\begin{equation}\label{diagramexactsequencea}
\begin{xy}*!C\xybox{%
\xymatrix{
1 \ar[r]  &  B_k(\surf{g}{, n})  \ar[d]^{\Phi_k}   \ar[r]    & B_{k,n}(\surf{g}{})  \ar[d]^{\rho_{k,n}} \ar[r]^{\psi_k}          &  B_n(\surf{g}{})   \ar[d]^{\rho_{n}} \ar[r]        & 1\\
1 \ar[r]  &   G_k\bigl(\surf{g}{}\bigr) \ar[r]           & B_{k,n}(\surf{g}{})/ \Gamma_3(B_{k,n}(\surf{g}{}))  \ar[r]^{\overline{\psi}_k} &  B_n(\surf{g}{})/ \Gamma_3(B_n(\surf{g}{}))  \ar[r]   & 1,}}
\end{xy}
\end{equation}
where  $\rho_{k,n}$ and $\rho_n$ denote the two canonical projections, $\Phi_k$ is the restriction of  $\rho_{k,n}$ to $B_k(\surf{g}{, n}) $, $\psi_k:B_{k,n}(\surf{g}{})\to B_n(\surf{g}{})$ is obtained geometrically by forgetting the first $k$ strings, $\overline{\psi}_k$ is the map induced by $\psi_k$ and $G_k\bigl(\surf{g}{}\bigr)$ is the kernel of~$\overline{\psi}_k$. We shall prove that $G_{\Sigma} = G_k\bigl(\surf{g}{}\bigr)$ and $\Phi_\Sigma = \Phi_k$. More precisely: 
\begin{thm}\label{main_thm} Let $k,n\geq 3$. There is a canonical isomorphism of groups~$\iota: G_\Sigma\to G_k\bigl(\surf{g}{}\bigr)$. Moreover one has $\iota\circ\Phi_\Sigma = \Phi_k$.
\end{thm}
Our second objective is to obtain rigidity results within a completely algebraic framework for some of the groups appearing in equation~(\ref{diagramexactsequencea}), thus extending those of~\cite{HK} mentioned above. 
\begin{thm}\label{main_thm_2}
Let $k,n\geq 3$. 
\begin{enumerate}[(i)] 
\item Denote by $\iota_{k,n}: B_{k,n} \to B_{k,n}(\surf{g}{})$, $\iota_k: B_k(\D_n) \to B_k(\surf{g}{, n}) $  and $\iota_n: B_{n} \to B_{n}(\surf{g}{})$ the natural inclusions (see~\cite{PR}). There exist injective homomorphisms $\gamma_k: G_k\to  G_k\bigl(\surf{g}{}\bigr)$, $\alpha_{k,n}:B_{k,n}/ \Gamma_2(B_{k,n})\to B_{k,n}(\surf{g}{})/ \Gamma_3(B_{k,n}(\surf{g}{}))$ and $\alpha_n: B_n/ \Gamma_2(B_n)  \to B_n(\surf{g}{})/ \Gamma_3(B_n(\surf{g}{}))$ so that following diagram of horizontal exact sequences is commutative:
\begin{equation}\label{diagramexactsequencethm}
\begin{xy}*!C\xybox{
\xymatrix{
1 \ar[r]  &  G_k  \ar[r]  \ar@/^3pc/@{-->}[ddd]^{\gamma_{k}}  & B_{k,n}/ \Gamma_2(B_{k,n})  \ar[r]     \ar@/^3pc/@{-->}[ddd]^{\alpha_{k,n}}  &  B_n/ \Gamma_2(B_n)  
\ar@/^3pc/@{-->}[ddd]^{\alpha_n}  \ar[r]    &  1\\
1 \ar[r]  &  B_k(\D_n)  \ar[d]^{\iota_k} \ar[u]_{p_k}   \ar[r]    & B_{k,n}  \ar[d]^{\iota_{k,n}}   \ar[u]_{r_{k,n}} \ar[r]^{}  &  B_n  \ar[d]^{\iota_{n}}   \ar[u]_{r_n} \ar[r]   & 1\\
1 \ar[r]  &  B_k(\surf{g}{, n})  \ar[d]^{\Phi_k}   \ar[r]    & B_{k,n}(\surf{g}{})  \ar[d]^{\rho_{k,n}} \ar[r]  &  B_n(\surf{g}{})   \ar[d]^{\rho_{n}} \ar[r]   & 1\\
1 \ar[r]  &  G_k\bigl(\surf{g}{}\bigr) \ar[r]  & B_{k,n}(\surf{g}{})/ \Gamma_3(B_{k,n}(\surf{g}{}))  \ar[r]^{}    &  B_n(\surf{g}{})/ \Gamma_3(B_n(\surf{g}{}))  \ar[r]   & 1,}}
\end{xy}
\end{equation}
where the rows are the short exact sequences of the commutative diagrams~(\ref{eq:diagartin}) and~(\ref{diagramexactsequencea}).

\item Let $G$ be a group, and $\Phi_G: B_k(\surf{g}{,n}) \to G$ be a surjective homomorphism whose restriction to the group~$B_k(\D_n)$ induces an injective homomorphism from $G_k$ to $G$. Then $G = G_k\bigl(\surf{g}{}\bigr)$ up to isomorphism; more precisely, $\Phi_G = \theta_G\circ\Phi_k$ where $\theta_G$ is an isomorphism.
  
\item Let $H$ be a group, and $\rho_H : B_{k,n}(\surf{g}{}) \to H$ be a surjective homomorphism whose restriction to the group~$B_{k,n}$ induces an injective homomorphism from $B_{k,n}/ \Gamma_2(B_{k,n})$ to $H$. 
 Then $H =  B_{k,n}(\surf{g}{})/{\Gamma_3(B_{k,n}(\surf{g}{}))}$ up to isomorphism; more precisely, $\rho_H = \theta_H\circ\rho_{k,n}$ where $\theta_H$ is an isomorphism.
  
\item Let $K$ be a group, and $\rho_K : B_n(\surf{g}{}) \to K$ be a surjective homomorphism whose restriction to the group~$B_{n}$ induces an injective homomorphism from $B_{n}/ \Gamma_2(B_{n})$ to $K$. 
Then $K=B_n(\surf{g}{})/{\Gamma_3(B_n(\surf{g}{}))}$ up to isomorphism; more precisely, $\rho_K = \theta_K\circ\rho_{n}$ where $\theta_K$ is an isomorphism. 
\end{enumerate}
\end{thm}

We will in fact obtain a stronger result, by proving that some of the above results remain true when the assumptions on $\Phi_G$, $\rho_K$ and $\rho_H$ are relaxed. Remarking that the group~$H_\Sigma$ of~\cite{HK} is a quotient of $B_{k,n}(\surf{g}{})/{\Gamma_3(B_{k,n}(\surf{g}{}))}$,  in Proposition~\ref{fondpropcons2}, we exhibit an alternative proof of~\cite[Theorem~4.3]{HK}.

\medskip

This paper is organised as follows. In Section~\ref{section2} we recall the definitions of (mixed) surface braid groups and their associated short exact sequences. The first part 
of Section~\ref{section3} is devoted to obtaining presentations of mixed surface braid groups. 
In Section~\ref{sec:presquotients} we describe the Abelianisations of (mixed) surface braid groups  
(Propositions \ref{lem:presgammaabel} and \ref{lem:presbngabel}) and we show in particular that it is not possible to embed 
$B_{k,n}/\Gamma_2(B_{k,n})$ in any Abelian quotient of $B_{k,n}(\surf{g}{})$.
In Section~\ref{sec:metaquotients}, we obtain similar results but at the level of quotients of surface mixed braid groups  by $\Gamma_{3}$ rather than by $\Gamma_{2}$.
We give a presentation for $B_{k,n}(\surf{g}{}) / \Gamma_3(B_{k,n}(\surf{g}{}))$ (Proposition \ref{lem:presgamma}) and a normal form for elements of this quotient (Corollary \ref{cor:decgammamix}), and we prove several 
rigidity results for (metabelian) quotients of  $B_{k,n}(\surf{g}{}) $ (Corollary \ref{cor:themain2iii}). In particular, we show that 
$B_{k,n}/\Gamma_2(B_{k,n})$ embeds in   $B_{k,n}(\surf{g}{}) / \Gamma_3(B_{k,n}(\surf{g}{}))$
and   we deduce Theorem~\ref{main_thm_2}(iii). Similar results are given in Section~\ref{sec:metaquotients2} for surface braid groups, the main result being Corollary~\ref{cor:themain2iii}, which implies Theorem~\ref{main_thm_2}(i) and (iv).

In Section~\ref{section4}, we prove  Theorem~\ref{main_thm} and Theorem~\ref{main_thm_2}(ii), and we show that  it is not possible to extend  the length function $\lambda: B_n \to \Z$ to braid groups of closed oriented surfaces (Proposition \ref{fondpropcons2}). Finally in Section~\ref{section:representations},
we describe  an algebraic approach to the Burau and Bigelow-Krammer-Lawrence representations that is based on the lower central series,  and
we explain why it is not possible to extend them to representations of surface braid groups. 
 This latter remark was made in \cite{HK} under certain conditions of a homological nature. Within a purely algebraic framework, we prove this non-existence result with fewer conditions than those given in~\cite{HK}.


\section{Preliminaries on configuration spaces} \label{section2}


Surface braid groups are a natural generalisation of both the classical
braid groups and the fundamental group of surfaces.
We recall Fox's definition in terms of
fundamental groups of configuration spaces~\cite{FoN}. Let $\Sigma$ be a connected surface, and let $\FF_n(\Sigma)=\Sigma^n
\setminus \Delta$, where $\Delta$ is  the set of $n$-tuples $(x_1,
\dots, x_n)\in \Sigma^n$ for which $x_i=x_j$ for some $i \not= j$. The
fundamental group $\pi_1(\FF_n(\Sigma))$ is called the \emph{pure
braid group} on $n$ strings of the surface $\Sigma$ and shall be
denoted by $P_n(\Sigma)$. The symmetric group $S_n$ acts freely on $\FF_n(\Sigma)$ by permutation of coordinates, and the fundamental group $\pi_1(\FF_n(\Sigma)/S_n)$ of the resulting quotient space, denoted by $B_n(\Sigma)$, is the \emph{braid group} on $n$ strings of the surface $\Sigma$. Further, $\FF_n(\Sigma)$ is a regular $n!$-fold covering of
$\FF_n(\Sigma)/S_n$, from which we obtain the following short exact sequence:
\begin{equation}\label{eq:permutation}
1\to P_n(\Sigma) \to B_n(\Sigma) \to S_{n}\to 1.
\end{equation}
In the case of the disc~$\D^{2}$, it is well known that $B_{n}(\D^2)\cong B_{n}$ and $P_{n}(\D^2)\cong P_{n}$.

Let $k,n\in\N$. Regarded as a subgroup of $S_{k+n}$, the group $S_k
\times S_n$ acts on $\FF_{k+n}(\Sigma)$. The fundamental group
$\pi_1\left(\FF_{k+n}(\Sigma)/(S_k \times S_n)\right)$ will be called
the \emph{mixed braid group of $\Sigma$ on $(k,n)$ strings}, and shall
be denoted by $\MB{k}{\Sigma}$.  We shall denote $\MB{k}{\D^2}$ simply by $B_{k,n}$. In an obvious manner, we have $P_{k+n}(\Sigma) \subset \MB{k}{\Sigma} \subset B_{k+n}(\Sigma)$. Mixed braid groups, which play an important r\^ole in~\cite{HK}, were defined previously in~\cite{GG2,Man,PR}, and were studied in more detail in~\cite{GG4} in the case where $\Sigma$ is the $2$-sphere $\mathbb{S}^2$. 

Consider the Fadell-Neuwirth fibration $\FF_{k+n}(\Sigma) \to \FF_{n}(\Sigma)$ given by forgetting the first $k$ coordinates. Its long exact sequence in homotopy yields the \emph{surface pure braid group} short exact sequence~\cite{FaN}:
\begin{equation}\label{eq:psequence}
1 \to P_k(\Sigma \setminus \sset{x_{1}}{x_{n}}) \to P_{k+n}(\Sigma)
\to P_n(\Sigma) \to 1,\tag{SPB}
\end{equation}
where $n\geq 3$ (resp.\ $n\geq 2$) if $\Sigma=\mathbb{S}^{2}$ (resp.\ $\Sigma$ is the projective plane $\mathbb{R}P^{2}$). In a similar manner, the map $\FF_{k+n}(\Sigma)/(S_k \times S_n) \to \FF_n(\Sigma)/S_n$ defined by forgetting the first $k$ coordinates, is a locally-trivial fibration whose fibre may be identified with
$\FF_k(\Sigma \setminus \sset{x_{1}}{x_{n}} )/S_{k}$.  With the same constraints on $n$ if $\Sigma=\mathbb{S}^{2}$ or $\mathbb{R}P^{2}$, this fibration gives rise to the \emph{surface mixed braid group} short exact sequence:
\begin{equation}\label{eq:sequence}
1 \to B_k(\Sigma \setminus \sset{x_{1}}{x_{n}}) \to \MB{k}{\Sigma}
\stackrel{\psi_{k}}{\to} B_n(\Sigma) \to 1,\tag{SMB}
\end{equation}
where $\psi_{k}$ is the epimorphism given in diagram \ref{diagramexactsequencea} that may be interpreted geometrically by forgetting the first $k$ strings. Note that~(\ref{eq:psequence}) is the restriction of~(\ref{eq:sequence}) to the corresponding pure braid groups. From now on, we denote $\D^2 \setminus \sset{x_{1}}{x_{n}}$ by $\D_n$.

The group~$B_k(\D_n)$ (resp.\ $P_k(\D_n)$) may be seen to be isomorphic to the subgroup of $B_{k+n}$ (resp.\ $P_{k+n}$) consisting of braids whose last $n$ strings are trivial (vertical). As we shall see in Section~\ref{section4}, one important ingredient in the construction of representations of the Artin braid groups is the splitting of the short exact sequences~(\ref{eq:sequence}) and~(\ref{eq:psequence}) when $\Sigma=\D^{2}$; in both cases, a section, which we refer to henceforth as the \emph{standard section}, is given by adding $k$ vertical strings (see for instance \cite{HK,B}). With the aim of obtaining representations of the braid groups of $\Sigma$, it is thus natural to ask in which cases these sequences split. Note that there is a commutative diagram of short exact sequences, where the first line is~(\ref{eq:psequence}), the second line is~(\ref{eq:sequence}), and the third line is:
\begin{equation*}
1 \to S_{k} \to S_{k}\times S_{n} \to S_{n} \to 1.
\end{equation*}
The question of the splitting of~(\ref{eq:psequence}) has been solved completely (see~\cite{GG6} for a summary). In particular, if $\Sigma$ is a compact surface without boundary and different from $\mathbb{S}^{2}$ and $\mathbb{R}P^{2}$ then~(\ref{eq:psequence}) only splits if $n=1$, and one may show that this implies the splitting of~(\ref{eq:sequence}) in this case. Using the methods of~\cite{GG1}, it follows that both~(\ref{eq:psequence}) and~(\ref{eq:sequence}) split if $\Sigma$ has non-empty boundary. Some partial results for the splitting of~(\ref{eq:sequence}) are known if $\Sigma$ has empty boundary (see for example~\cite{FaV,GG4} for the case of $\mathbb{S}^2$), but in general the question remains unanswered.

\section{Surface braid groups and lower central series}
 \label{section3}

As in the Introduction, let $\surf{g}{}$ be a compact, connected orientable surface
of genus $g\geq 0$ with a single boundary component, and let $k\geq 1$ and $n\geq 0$. We will make use of the notation introduced in the commutative diagrams~(\ref{diagramexactsequencea}) and~(\ref{diagramexactsequencethm}). We focus on $\surf{g}{}$ essentially for two reasons; the first is that as we mentioned in Section~\ref{section2}, the short exact sequence~(\ref{eq:sequence}) for $\surf{g}{}$ splits and therefore $B_n(\surf{g}{})$ acts by conjugation on $B_k(\surf{g}{, n})$. The second reason is that Theorem~\ref{main_thm_2}(iv) is not valid if we replace $\surf{g}{}$ by a compact surface without boundary (this fact will be a straightforward consequence of Proposition~\ref{fondpropcons2}).

In this section, we shall prove parts~(i), (iii) and~(iv) of Theorem~\ref{main_thm_2}.
Taking into account the commutative diagrams~(\ref{eq:diagartin}) and~(\ref{diagramexactsequencea}) as well as the following commutative diagram:
\begin{equation} \label{eq:injnat}
\begin{xy}*!C\xybox{%
\xymatrix{
1 \ar[r]  &  B_k(\D_n)  \ar[d]^{\iota_k}    \ar[r]    & B_{k,n}  \ar[d]^{\iota_{k,n}}   \ar[r]        &  B_n  \ar[d]^{\iota_{n}}    \ar[r]                    & 1\\
1 \ar[r]  &  B_k(\surf{g}{, n})      \ar[r]    & B_{k,n}(\surf{g}{})    \ar[r]                                &  B_n(\surf{g}{})     \ar[r]                      & 1,}
}
\end{xy}
\end{equation}
to prove Theorem~\ref{main_thm_2}(i), it will suffice to show the existence of the homomorphisms $\gamma_k, \alpha_k, \alpha_{k,n}$ and $\alpha_{n}$, and to verify commutativity in the vertical parts of the commutative diagram~(\ref{diagramexactsequencethm}). The main result of this section is Proposition~\ref{prop:techresult}, which is a stronger version of Theorem~\ref{main_thm_2}(iii). In Section~\ref{sec:mixed}, we start by exhibiting a presentation for $\bsurf{k,n}{g}{}$. 
In what follows, we will consider the following disjoint sets:  
$$\begin{array}{lcllcl}S&=&\{ \si_1, \ldots, \si_{k-1}\},& \widetilde{S}&=&\{\tsi_1, \ldots, \tsi_{n-1}\},\\
AB&=&\{a_1, b_1, \ldots, a_g, b_g\},&\widetilde{AB}&=&\{\tia_1, \tib_1, \ldots, \tia_g, \tib_g\},\\
Z&=&\{\zeta_1,\ldots, \zeta_n\}.
\end{array}$$
If $k=1$ (resp.\ $n=0$, $n=1$, $g=0$) then $S$ (resp.\ $\widetilde{S}\cup Z$, $\widetilde{S}$, $AB\cup \widetilde{AB}$) is taken to be empty. For $c,d\in AB$ we write $c<d$ if $c \in \{a_i,b_i\}$ and $d \in\{a_j,b_j\}$ with $i<j$. Similarly, for $c,d\in \widetilde{AB}$ we write $c<d$ if $c \in\{\tia_i,\tib_i\}$ and $d \in\{\tia_j,\tib_j\}$ with $i<j$. If $x$ and $y$ are elements of a group then we set $x^y=y^{-1}xy$ and $[x,y] = xyx^{-1} y^{-1}$.

\subsection{Presentations of surface mixed braid groups}\label{sec:mixed}


In order to prove Theorem~\ref{main_thm_2}, we will need to understand the structure of surface (mixed) braid groups and some of their quotients. With this in mind, in this section, we recall presentations of $\bsurf{k}{g}{,n}$ and $\bsurf{n}{g}{}$, and we derive a presentation of $\bsurf{k,n}{g}{}$.

If $g=0$, part~(i) of the following result is proved in~\cite{Lam}. If $g\geq 1$, the presentations of parts~(i) and~(ii) may be found in~\cite{B}. From hereon, if $g=0$ then the elements of $AB \cup \widetilde{AB}$ and all relations containing these elements should be suppressed.

\begin{prop}\label{presconnues}
Let $k,n\geq 1$, and let $g\geq 0$.
\begin{enumerate}[(i)]  
\item The group~$\bsurf{k}{g}{,n}$ admits the following group presentation:

\noindent \textbf{Generating set:}  $S\cup AB\cup Z$; 

\noindent  \textbf{Relations:}
$$\begin{array}{cll}
{\textit{(a.1)}}&\si_i\si_j=\si_j\si_i,&\lvert i-j \rvert \geq 2;\\
\textit{(a.2)}&\si_i\si_{i+1}\si_i= \si_{i+1}\si_i \si_{i+1},& 1\leq i\leq k-2;\\
\textit{(a.3)}&c\si_i= \si_{i} c,&i\not=1,\; c \in AB\cup Z;\\
\textit{(a.4)}&c \si_1 c \si_1= \sigma_1 c \sigma_1 c, & c \in AB\cup Z; \\
\textit{(a.5)}&a_i \sigma_1 b_i = \sigma_1 b_i \sigma_1 a_i \sigma_1,& i\in \{1, \ldots, g\};\\
\textit{(a.6)}&(\sigma_1^{-1} c \sigma_1) d=d (\sigma_1^{-1} c \sigma_1), & c,d\in AB,\   c<d;\\
\textit{(a.7)}&(\si_1^{-1}\zeta_{i} \si_1)c = c (\si_1^{-1}\zeta_{i}\si_1), &c \in AB, \ \zeta_i \in  Z;\\
\textit{(a.8)}&(\si_1^{-1}\zeta_{i}\si_1)\zeta_{j} = \zeta_{j} (\si_1^{-1}\zeta_{i}\si_1),& i<j.
\end{array}$$
\item The group~$\bsurf{n}{g}{}$ admits the following group presentation:

\noindent \textbf{Generating set:}  $\widetilde{S}\cup \widetilde{AB}$; 

\noindent  \textbf{Relations:}
$$\begin{array}{cll}
\textit{(b.1)}&\tsi_i\tsi_j=\tsi_j\tsi_i,&\lvert i-j \rvert \geq 2;\\
\textit{(b.2)}&\tsi_i\tsi_{i+1}\tsi_i= \tsi_{i+1}\tsi_i \tsi_{i+1},& 1\leq i\leq n-2;\\
\textit{(b.3)}&\tic\,\tsi_i= \tsi_i \tic,&i\not=1,\; \tic \in \widetilde{AB};\\
\textit{(b.4)}&\tic \,\tsi_1 \tic \,\tsi_1= \tsi_1 \tic \,\tsi_1 \tic,& \tic \in \widetilde{AB}; \\
\textit{(b.5)}&\tia_i\tsi_1 \tib_i = \tsi_1 \tib_i\tsi_1 \tia_i \tsi_1,& i\in \{1, \ldots, g\};\\
\textit{(b.6)}&(\tsi_1^{-1} \tic\, \tsi_1) \tilde{d}=\tilde{d} (\tsi_1^{-1}  \tic \,\tsi_1),&\tic,\widetilde{d} \in \widetilde{AB},\   \tic<\widetilde{d}.
\end{array}$$
\end{enumerate}
\end{prop}

The presentation of part~(i) may be adapted to the case~$n= 0$ by suppressing $Z$ in the presentation of part~(ii). However, to obtain a presentation of $\bsurf{k,n}{g}{}$ from Proposition~\ref{presconnues}, it will be convenient to have both presentations at our disposal. If $g=0$ then $\bsurf{k,n}{g}{}$ is equal to $B_{k,n}$. An alternative presentation of $\bsurf{k,n}{g}{}$ may be found in~\cite{HK}; in the special case of $B_{k,n}$, see~\cite{FRZ}.

\begin{prop}\label{prop:presMB}
Let $k,n\geq 1$, and let $g\geq 0$. The group~$\bsurf{k,n}{g}{}$ admits the following group presentation:

\noindent \textbf{Generating set:} $\Omega_{k,n} = S\cup \widetilde{S}\cup AB\cup \widetilde{AB}\cup Z$;

\noindent  \textbf{Relations:}
\begin{enumerate}[(a)]
\item the relations~(a.1)--(a.8) given in Proposition~\ref{presconnues}(i). 
\item the relations (b.1)--(b.6) given in Proposition~\ref{presconnues}(ii).
\item the relations that describe the action of~$\bsurf{n}{g}{}$ on~$\bsurf{k}{g}{,n}$:
\begin{enumerate}
\item[(c.1)] $\tsi_i\si_j\tsi_i^{-1} = \tia_i\si_j\tia_i^{-1} = \tib_i\si_j\tib_i^{-1} =\si_j$;

\item[(c.2)]  $\tsi_ia_j\tsi_i^{-1} = a_j$, $\tsi_ib_j\tsi_i^{-1} = b_j$; 

\item[(c.3)] $\left\{\begin{array}{clll}
\textit{(c.3.1)}&\tsi_i\z_{i+1}\tsi_i^{-1} = \z_i;\\
\textit{(c.3.2)}&\tsi_i\z_{i}\tsi_i^{-1} = \z_i^{-1}\z_{i+1}\z_i;\\
\textit{(c.3.3)}& \tsi_i\z_{j}\tsi_i^{-1} = \z_j,&j\neq i,i+1;\end{array}\right.$
(c.4) $\left\{\begin{array}{cllll}
\textit{(c.4.1)}&\tia_i \z_1 \tia_i^{-1} =\z_1^{a_i \z_1} ;\\
\textit{(c.4.2)}&\tib_i \z_1 \tib_i^{-1} =  \z_1^{b_i \z_1};\\
\textit{(c.4.3)}&\tia_i \z_j \tia_i^{-1} = \z_j^{[a_i^{-1},\z_1^{-1}]}, &j\neq 1;\\
\textit{(c.4.4)}&\tib_i \z_j \tib_i^{-1} = \z_j^{[b_i^{-1},\z_1^{-1}]}, &j\neq 1;
\end{array}\right.$

\item[(c.5)] $\left\{\begin{array}{cll}
\textit{(c.5.1)}& \tia_ia_i\tia_i^{-1} = \z_1^{-1} a_i \z_1;\\ 
\textit{(c.5.2)}&\tia_i a_j\tia_i^{-1} =  a_j^{[a_i^{-1},\z_1^{-1}]},&i > j;\\
\textit{(c.5.3)}&\tia_i a_j \tia_i^{-1} = a_j,&j > i;\end{array}\right.$  
\ \ \ \ \ \ (c.6) $\left\{\begin{array}{cll}
\textit{(c.6.1)}& \tib_i b_i\tib_i^{-1} = \z_1^{-1} b_i \z_1;\\ 
\textit{(c.6.2)}&\tib_i b_j\tib_i^{-1} =  b_j^{[b_i^{-1},\z_1^{-1}]}, &i > j;\\
\textit{(c.6.3)}&\tib_i b_j \tib_i^{-1} = b_j,&j > i;\end{array}\right.$

\item[(c.7)] $\left\{\begin{array}{cll} 
\textit{(c.7.1)}&\tia_ib_i\tia_i^{-1} = b_i \z_1;\\  
\textit{(c.7.2)}&\tia_i b_j\tia_i^{-1} = b_j^{[a_i^{-1},\z_1^{-1}]},  &i > j;\\
\textit{(c.7.3)}&\tia_i b_j \tia_i^{-1} = b_j,&j > i
\end{array}\right.$ 
\ \ \ \ \ \ \ \ \ \ (c.8) $\left\{\begin{array}{cll}
 \textit{(c.8.1)}&\tib_ia_i\tib_i^{-1} = \z_1^{-1} a_i [b_i^{-1},\z_1^{-1}];\\
\textit{(c.8.2)}&\tib_i a_j\tib_i^{-1} = a_j^{[b_i^{-1},\z_1^{-1}]},&\hspace*{-0.9cm}i > j;\\ 
\textit{(c.8.3)}&\tib_i a_j \tib_i^{-1} = a_j,&\hspace*{-0.9cm}j > i. \end{array}\right.$
\end{enumerate}
\end{enumerate}
\end{prop}

\begin{proof}
If $g\geq 1$, we first give a geometric interpretation of the generators of $\bsurf{k,n}{g}{}$. We represent $\surf{g}{}$ as a regular polygon with $4g$ edges, equipped with the standard identification of edges, and one boundary component.
We consider braids to be paths on the polygon,
which we draw with the usual over- and under-crossings,
and we interpret the braids depicted in Figure~\ref{figure1mixed}
 as geometric representatives of the generators of $\bsurf{k}{g}{,n}$, and those depicted in Figure~\ref{figure2mixed} as the coset representatives of generators of $\bsurf{n}{g}{}$ in $\bsurf{k,n}{g}{}$. For example, 
for the braid $a_r$ (respectively $b_r$), the only non-trivial string is the first one,
which passes  through the wall $\alpha_r$ (respectively the wall $\beta_r$). If $g\geq 0$,
one can therefore check that relations hold for corresponding geometric braids, see Figure~\ref{figure3mixed} for example.  
 
From Section~\ref{section2}, the short exact sequence~(\ref{eq:sequence}) 
splits. The group~$\bsurf{k,n}{g}{}$ is thus isomorphic to a semi-direct product~$\bsurf{n}{g}{}\ltimes \bsurf{k}{g}{,n}$. Proposition~\ref{presconnues} implies that the set of relations~(a) and~(b) provide a complete set of relations for $\bsurf{k}{g}{,n}$ and $\bsurf{n}{g}{}$ respectively. The set of relations~(c) describes the action of the generators of $\bsurf{k}{g}{,n}$ on those of $\bsurf{n}{g}{}$. The set of relations~(a)--(c) therefore form a complete set of relations for $\bsurf{k,n}{g}{}$ by~\cite[Chap.~10, Proposition~1]{joh}.
\end{proof}

\begin{figure}[h]
 \centering
\includegraphics[width=10.5cm]{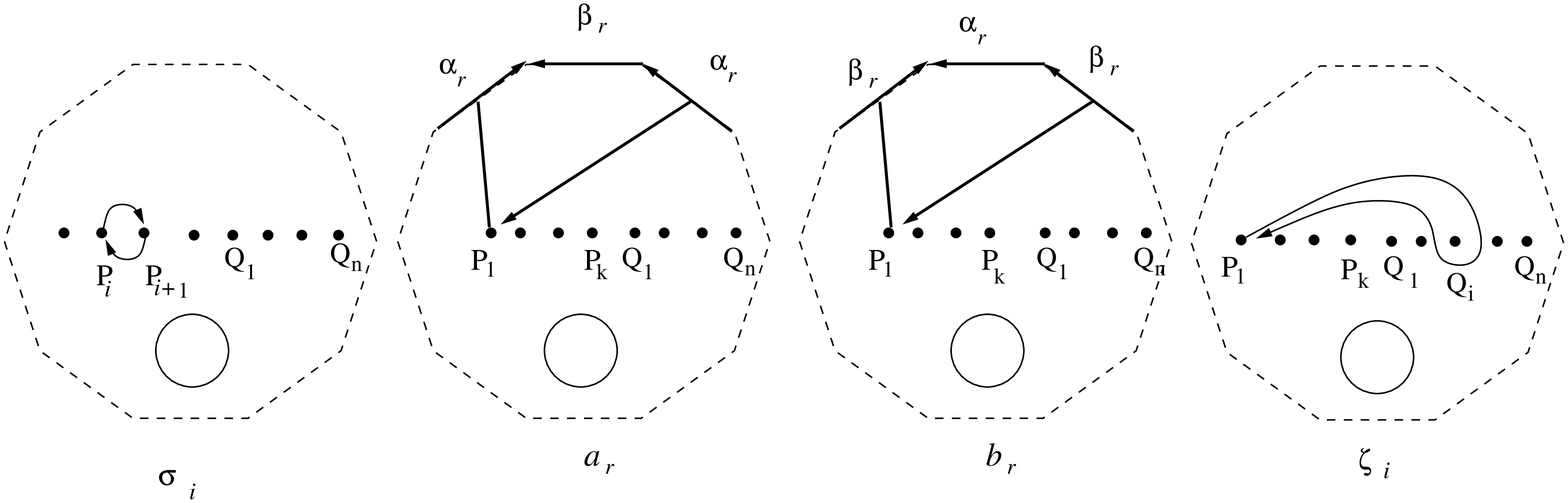}
 \caption{The generators $\si_1, \ldots, \si_{k-1}, a_1, b_1, \ldots, a_g, b_g,\zeta_1,\ldots, \zeta_n$}\label{figure1mixed}
\end{figure}

\begin{figure}[h]
 \centering
\includegraphics[width=8.6cm]{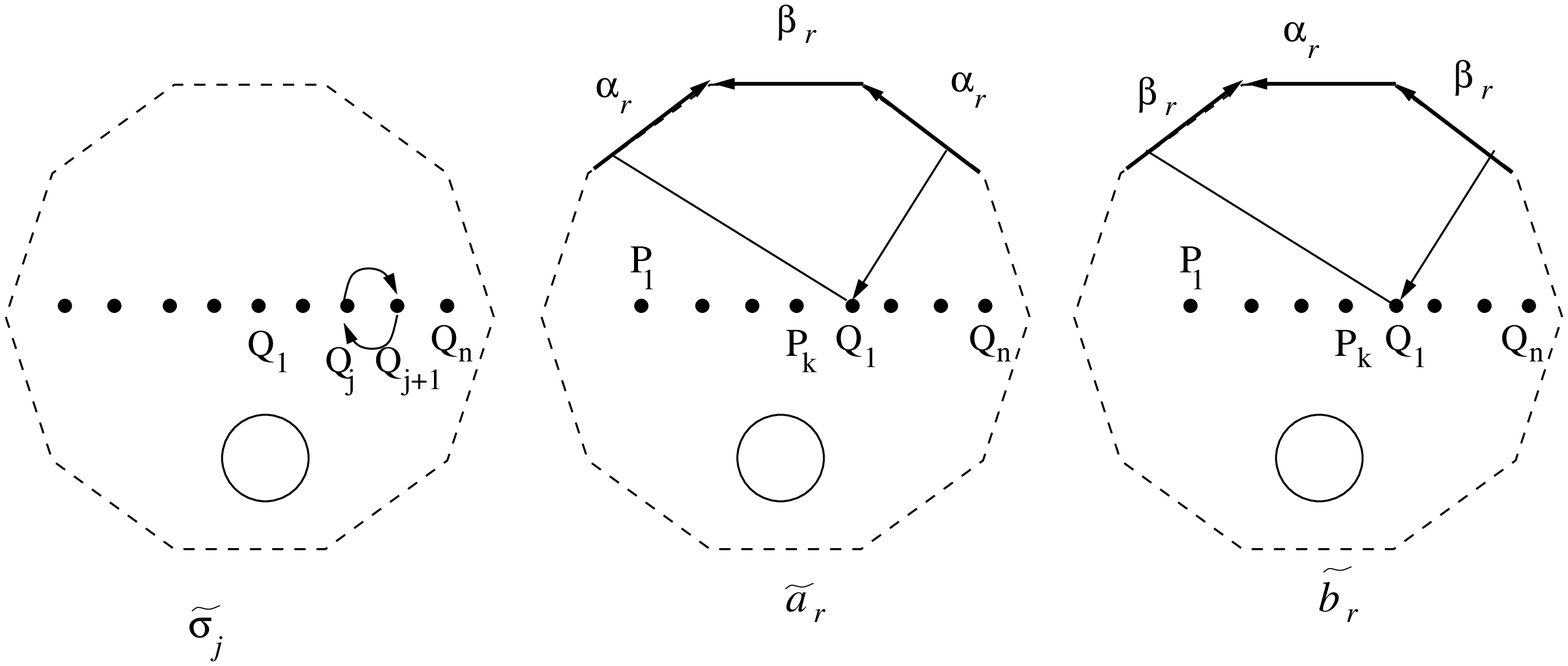}
 \caption{The generators $\tsi_1, \ldots, \tsi_{n-1}, \tia_1, \tib_1, \ldots, \tia_g, \tib_g$}\label{figure2mixed}
\end{figure}

\begin{figure}
 \centering
\includegraphics[width=9.7cm]{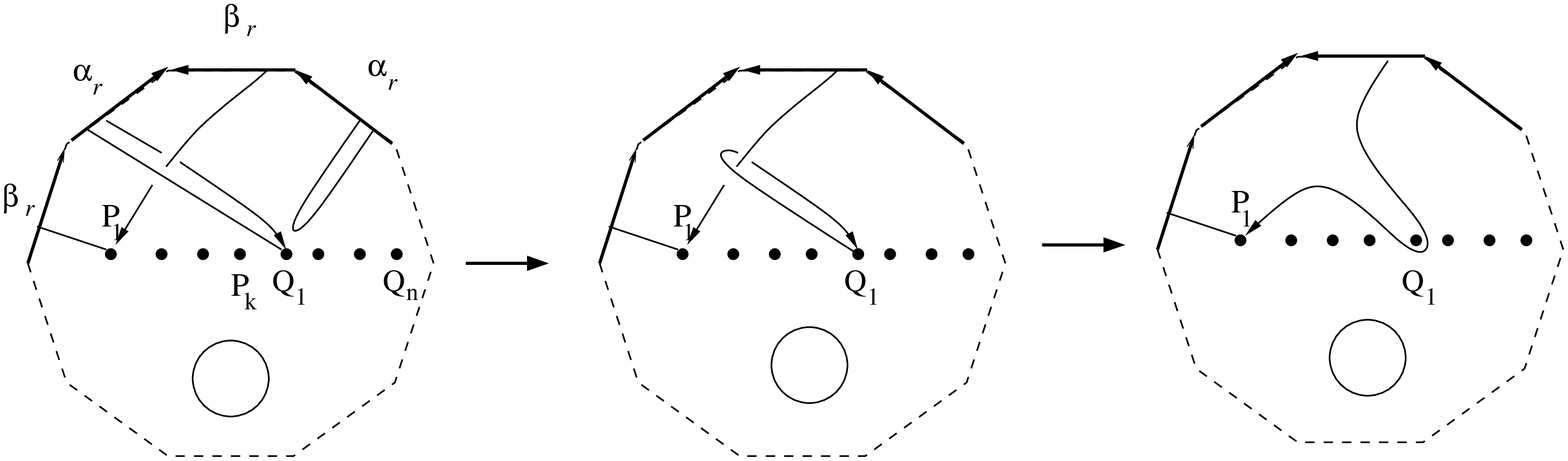}
 \caption{The braids $\tia_ib_i\tia_i^{-1}$ and  $b_i \z_1$ are isotopic.}\label{figure3mixed}
\end{figure}

\subsection{Abelian quotients of surface mixed braid groups}\label{sec:presquotients}

In this section, we use Propositions~\ref{presconnues} and~\ref{prop:presMB} to describe the Abelianisations of the (mixed) surface braid groups that arise in our study. We start with the case of mixed surface braid groups. We believe that the case $g=0$ is well known to the experts in the field, but since we were not able to find a reference in the literature, we provide a short proof.

\begin{prop}\label{lem:presgammaabel}
Let $n,k\geq 1$ and $g\geq 0$, let $\widehat{r}_{k,n}: \bsurf{k,n}{g}{}\to \bsurf{k,n}{g}{}/\Gamma_2(\bsurf{k,n}{g}{})$ denote the canonical projection (if $g=0$ then $\widehat{r}_{k,n}$ is the homomorphism $r_{k,n}$), and let
\begin{equation*}
\widehat{S} = \begin{cases}
\varnothing & \text{if $n=k=1$}\\
\brak{\si} & \text{if $k\geq 2$ and $n=1$}\\
\brak{\tsi} & \text{if $k=1$ and $n\geq 2$}\\
\brak{\si,\tsi} & \text{if $k,n\geq 2$}
\end{cases} \quad \text{and} \quad \widehat{Z} =
\begin{cases}
\brak{\z} & \text{if $g=0$}\\
\varnothing & \text{if $g\geq 1$.}
\end{cases}
\end{equation*}
Then the group $\bsurf{k,n}{g}{}/\Gamma_2(\bsurf{k,n}{g}{})$ admits the following group presentation:\\
\textbf{Generating set:} $\widehat{r}_{k,n}(S\cup\widetilde{S}\cup Z\cup AB\cup\widetilde{AB})=\widehat{S}\cup \widehat{Z} \cup AB \cup\widetilde{AB}$;\\
\textbf{Relations:}
\begin{enumerate}[\textbullet]
\item $xy = yx$ for $x,y\in \widehat{S}\cup \widehat{Z} \cup AB \cup\widetilde{AB}$, $x\neq y$;
\item if $g\geq 1$ then $s^2 = 1$ for all $s\in \widehat{S}$.
\end{enumerate}
In particular, $\widehat{r}_{k,n}(S)= \{\si\}$ if $k\geq 2$, $\widehat{r}_{k,n}(\widetilde{S})= \{\tsi\}$ if $n\geq 2$, $\widehat{r}_{k,n}(Z)= \{\z \}$ if $g=0$, and the group~$\bsurf{k,n}{g}{}/\Gamma_2(\bsurf{k,n}{g}{})$ is isomorphic to $\Z^{\lvert \widehat{S}\rvert +1}$ if $g=0$, and to 
$\Z^{4g}\times \Z_2^{\lvert\widehat{S}\rvert}$ if $g\geq 1$, where $\lvert\widehat{S}\rvert$ denotes the cardinal of $\widehat{S}$.
\end{prop}

\begin{proof}
Consider the presentation of $\bsurf{k,n}{g}{}$ given by Proposition~\ref{prop:presMB}. To obtain a presentation of the quotient~$\bsurf{k,n}{g}{}/\Gamma_2(\bsurf{k,n}{g}{})$, we must add the relations of the form $xy = yx$ for all $x,y$ in $\Omega_{k,n}$ to the presentation of $\bsurf{k,n}{g}{}$. First suppose that $g=0$. Then none of the relations~(c) of Proposition~\ref{prop:presMB} exist, with the exception of relation~(c.3). The group $B_{k,n}=\bsurf{k,n}{0}{}$ is generated by $S\cup \widetilde{S} \cup Z$. Under Abelianisation, if $k\geq 2$ (resp.\ $n\geq 2$), the elements of $S$ (resp.\ $\widetilde{S}$, $Z$) are identified to a single element $\si$ (resp.\ $\tsi$, $\z$) by relation~(a.2) (resp.\ relation~(b.2), relation~(c.3)), so the given  generating set of $B_{k,n}$ reduces to the generating set $\widehat{S}\cup \widehat{Z}$ of $B_{k,n}/\Gamma_2(B_{k,n})$, and the only defining relations are commutation relations. So the statement holds in the case $g=0$. Assume now that $g\geq 1$. For the elements of $\widehat{S}\cup \widehat{Z}$, the same analysis holds as in the case $g=0$. Additionally, relation~(c.7.1) implies that $\widehat{r}_{k,n}(\z_{1})$ is trivial, so $\widehat{Z}=\varnothing$, and thus  the relations of type~(c) do not give any extra information in $\bsurf{k,n}{g}{}/\Gamma_2(\bsurf{k,n}{g}{})$. Under $\widehat{r}_{k,n}$, if $k\geq 2$ (resp.\ $n\geq 2$), relations~(a.1)--(a.8) (resp.\ relations~(b.1)--(b.6)) do not yield any new relations, with the exception of relation~(a.5) (resp.\ relation~(b.5)) that reduces to $\si^2 = 1$ (resp.\ $\tsi^2 = 1$) in $\bsurf{k,n}{g}{}/\Gamma_2(\bsurf{k,n}{g}{})$, and the result follows. If $k=1$ (resp.\ $n=1$) then relations~(a.1)--(a.8) (resp.\ relations~(b.1)--(b.6)) do not exist, and again we see that the statement holds, which completes the proof of the case $g\geq 1$.
\end{proof}


For $k,g\geq 1$ and $n\geq 0$, let $\widehat{r}_{k}: \bsurf{k}{g,n}{}\to \bsurf{k}{g,n}{}/\Gamma_2(\bsurf{k}{g,n}{})$ denote the canonical projection. It will be convenient to denote the disc $\D^2$ by $\D_0$, so $B_{k}(\D_0)=B_{k}$. Note once more that if $k=1$ (resp.\ $n=0$, $g=0$) then all references to the element $\si$ (resp.\ to the set $Z$, to the set $AB$) in the generators and relations should be suppressed.

\begin{prop}\label{lem:presbngabel} 
Let  $k\geq 1$ and $g,n\geq 0$. 
The group~$\bsurf{k}{g,n}{}/\Gamma_2(\bsurf{k}{g,n}{})$ admits the following presentation:\\
\noindent \textbf{Generating set:} $\widehat{r}_{k}(S\cup AB\cup Z)= \{\si\} \cup AB \cup Z$;\\
\noindent  \textbf{Relations:}
\begin{enumerate}[\textbullet]
\item $xy = yx$ for all $x,y\in \{\si\}\cup AB\cup Z$, $x\neq y$;
\item if $g\geq 1$ then $\si^2 = 1$.
\end{enumerate}
Consequently, $B_k(\D_n)/\Gamma_2(B_k(\D_n))$ is free Abelian on the set $\widehat{r}_{k}(S\cup Z)=\{\si\} \cup Z$, and if $g\geq 1$, the group $\bsurf{k}{g,n}{}/\Gamma_2(\bsurf{k}{g,n}{})$ is isomorphic to $\Z^{2g+n}\times \Z_2$ if $k\geq 2$ and to $\Z^{2g+n}$ if $k=1$.
\end{prop}

If $n=0$, the result may be found in~\cite{BF}. The proof of Proposition~\ref{lem:presbngabel} follows in an similar manner to that of Proposition~\ref{lem:presgammaabel} using the presentation of $\bsurf{k}{g,n}{}$ given by Proposition~\ref{presconnues}. 


For $k\geq 1$ and $n\geq 0$, let $r_{k}: B_k(\D_n) \to B_k(\D_n)/\Gamma_2(B_{k}(\D_n))$ denote the canonical projection. We now apply Propositions~\ref{lem:presgammaabel} and~\ref{lem:presbngabel} to analyse the homomorphisms $\bar\iota_{k,n}: B_{k,n}/\Gamma_2(B_{k,n})\to B_{k,n}(\surf{g}{})/\Gamma_2(B_{k,n}(\surf{g}{}))$ and $\bar\iota_k: B_k(\D_n)/\Gamma_2(B_{k}(\D_n))\to\bsurf{k}{g,n}{}/\Gamma_2(\bsurf{k}{g,n}{})$ induced by the homomorphisms $\iota_{k,n}:B_{k,n}\to B_{k,n}(\surf{g}{})$ and $\iota_{k}: B_k(\D_n)\to\bsurf{k}{g,n}{}$ on the corresponding Abelianisations. These homomorphisms thus satisfy the relations $\widehat{r}_{k,n}\circ\iota_{k,n}= \bar\iota_{k,n} \circ r_{k,n}$ and $\widehat{r}_{k}\circ\iota_{k}= \bar\iota_{k} \circ r_{k}$.   


\begin{cor} \label{nolinearextK1}
Let $g,k\geq 1$, and let $G$ be an Abelian group.
\begin{enumerate}[(i)]
\item Let $n\geq 1$, and let $P_G: B_{k,n}(\surf{g}{}) \to G$ and $\overline{P_{G}}: B_{k,n}/\Gamma_2(B_{k,n}) \to G$ be homomorphisms for which~$P_G\circ \iota_{k,n}= \overline{P_{G}} \circ r_{k,n}$. Then $\overline{P_{G}}$ is not injective. In particular, the homomorphism $\bar\iota_{k,n}$ is not injective.

\item Let $n\geq 0$.
\begin{enumerate}[(a)]
\item If $k=1$ then the homomorphism $\bar\iota_{k}$ is injective.
\item Suppose that $k\geq 2$. Let $P_G: \bsurf{k}{g,n}{} \to G$ and $\overline{P_{G}}: B_k(\D_n)/\Gamma_2(B_{k}(\D_n))\to G$ be homomorphisms for which~$P_G\circ \iota_{k}= \overline{P_{G}} \circ r_{k}$. Then $\overline{P_{G}}$ is not injective. In particular, the homomorphism $\bar\iota_{k}$ is not injective.
\end{enumerate}
\end{enumerate}
%
%
%
%
\end{cor}

\begin{proof}\mbox{}
\begin{enumerate}[(i)]
\item We start by proving the non-injectivity of $\bar\iota_{k,n}$. Suppose that $G=\bsurf{k,n}{g}{}/\Gamma_2(\bsurf{k,n}{g}{})$ and $P_{G}=\widehat{r}_{k,n}$, and let $\overline{P_{G}}=\bar\iota_{k,n}$. Thus $\widehat{r}_{k,n}\circ\iota_{k,n}= \bar\iota_{k,n} \circ r_{k,n}$, and so $\bar\iota_{k,n}$ cannot be injective because $r_{k,n}(\z_1)\neq 1$ and $\widehat{r}_{k,n}(\z_1) = 1$ by Proposition~\ref{lem:presgammaabel}. We now consider the general case. Let $P_{G}$ and $\overline{P_{G}}$ be as in the statement. Since $G$ is Abelian, $P_{G}$ factors through $\widehat{r}_{k,n}$, so as in the case of $\bar\iota_{k,n}$, we have $P_{G}(\z_1)=1$ and $r_{k,n}(\z_1)\neq 1$, which implies the non-injectivity of $\overline{P_{G}}$ using the relation $P_G\circ \iota_{k,n}= \overline{P_{G}} \circ r_{k,n}$.

\item
\begin{enumerate}[(a)]
\item If $k=1$ then Proposition~\ref{lem:presbngabel} implies that $B_k(\D_n)/\Gamma_2(B_{k}(\D_n))$ is a free Abelian group generated by $Z$ ($Z$ may be empty if $n=0$), and that the homomorphism $\bar{\iota}_k$ identifies~$B_k(\D_n)/\Gamma_2(B_{k}(\D_n))$ with the direct factor of $\bsurf{k}{g,n}{}/\Gamma_2(\bsurf{k}{g,n}{})$ generated by $Z$, so $\bar{\iota}_k$ is injective.

\item If $k\geq 2$ then the argument is similar to that of part~(i), where we replace $\z_{1}$ by $\si_{1}^{2}$.\qedhere
\end{enumerate}
\end{enumerate}
\end{proof}

\subsection{Metabelian quotients of surface mixed braid groups}\label{sec:metaquotients}


In this section, the aim is to obtain results similar to those of Section~\ref{sec:presquotients}, but on the level of quotients by $\Gamma_{3}$ rather than by $\Gamma_{2}$. With Section~\ref{section4} in mind, our principal interest is in detecting the differences between these two types of quotient. 

The following result will provide the crucial argument in the proof of Theorems~\ref{main_thm} and~\ref{main_thm_2}.

\begin{prop} \label{prop:techresult}
Let  $k,n\geq 3$, and let $g\geq 0$. Let $H$ be a group, and let $\rho_H : B_{k,n}(\surf{g}{}) \to H$ be a surjective homomorphism. Consider the presentation of $B_{k,n}(\surf{g}{})$ given in Proposition~\ref{prop:presMB}, and let $R$ be a set of words on~$\Omega_{k,n}\cup \Omega_{k,n}^{-1}$ whose normal closure in $B_{k,n}(\surf{g}{})$ is equal to $\ker{(\rho_{H})}$. Assume that there exist $\si,\tsi$ in $H$ such that $$\rho_{H}(S)= \{\si\} \textit{ and } \rho_{H}(\widetilde{S})= \{\tsi\}.$$ Then the following assertions hold.
\begin{enumerate}[(i)]
\item There exists $\z\in H$ such that $\rho_{H}(Z)= \{\z\}$.
\item Let $R_\Omega$ denote the set of words obtained from $R$ by replacing each of the letters $\si_i$, $\tsi_i$ and $\z_i$ by $\si, \tsi$ and $\z$ respectively (and their inverses by $\si^{-1}$, $\tsi^{-1}$ and $\z^{-1}$ respectively). Then the group~$H$ possesses the following presentation:\\ \textbf{Generating set:} $\Omega =\rho_{H}(\Omega_{k,n}) =\{ \si, \tsi, \zeta\}\cup AB \cup\widetilde{AB}$;\\
\textbf{Relations:}
\begin{enumerate}[\textbullet]
\item $w = 1$ for all $w\in R_\Omega$;
\item $xy = yx$ for all $x,y\in \Omega$, $x\neq y$, and $\{x,y\}\notin \left\{\{a_i, b_i\},  \{\tia_i,\tib_i\}, \{b_i,\tia_i\}, \{\tib_i,a_i\} ; i=1,
\ldots, g\right\}$;
\item $[a_i,  b_i]=\sigma^2; \quad [\tia_i,  \tib_i]=\tsi^2; \quad  [a_i,  \tib_i]=[\tia_i,b_i]= \zeta \text{ for  } i=1,
\ldots, g$.
\end{enumerate}
\end{enumerate}
\end{prop}

In the above statement, we recall once more that if $g=0$ then $\Omega=\{ \si, \tsi, \zeta\}$, and that any relations involving the elements of the set $AB\cup \widetilde{AB}$ should be suppressed.

\pagebreak

\begin{proof}\mbox{}
\begin{enumerate}[(i)]
\item We analyse the images under $\rho_{H}$ of the relations of $B_{k,n}(\surf{g}{})$ given by the presentation of Proposition~\ref{prop:presMB}. Under $\rho_{H}$, for all $3\leq j\leq n$, relation~(c.3.3) with $i=j-2$ becomes $\tsi \rho_{H}(\z_{j}) \tsi^{-1}= \rho_{H}(\z_{j})$ and relation~(c.3.1) with $i=j-1$ becomes $\tsi \rho_{H}(\z_{j}) \tsi^{-1}= \rho_{H}(\z_{j-1})$, hence $\rho_{H}(\z_{j})=\rho_{H}(\z_{j-1})$, and thus $\rho_{H}(\z_{2})=\cdots=\rho_{H}(\z_{n})$. Relation~(c.3.2) with $i=1$ yields $\tsi \rho_{H}(\z_{1})\tsi^{-1} = \rho_{H}(\z_{1})^{-1}\rho_{H}(\z_{2})\rho_{H}(\z_{1})$, and relation~(c.3.3) with $i=2$ and $j=1$ (recall that $n\geq 3$) becomes $\tsi\rho_{H}(\z_{1})\tsi^{-1} =\rho_{H}(\z_{1})$. Thus $\rho_{H}(\z_{1})=\rho_{H}(\z_{2})$, and all of the $\z_i$ have a common image~$\z$ under $\rho_{H}$.

\item By the assumption on the set $R$ and the presentation of~$\bsurf{k,n}{g}{}$, the group~$H$ has a group presentation with $\Omega_{k,n}$ as a generating set, and whose defining relations are obtained by adding the relations of the form $w = 1$ for all $w$ in $R$ to those given in the presentation of Proposition~\ref{prop:presMB}. But by hypothesis, $\rho_{H}(S)= \{\si\}$ and $\rho_{H}(\widetilde{S})= \{\tsi\}$, so $\rho_{H}(Z)= \{\z\}$ by part~(i). Hence we obtain a new presentation of $H$ by replacing $\Omega_{k,n}$ by $\Omega$ and $\si_i$, $\tsi_{i}$ and $\z_{i}$ by $\si$, $\tsi$ and $\z$ respectively in all of the defining relations. Let us show that these relations reduce to those given in the statement. First, the relations obtained from~(a.1), (a.2), (b.1) and~(b.2) may be removed since they are satisfied trivially. The relations~(c.3) become $[\tsi,\z] = 1$. Since $k,n\geq 3$, relations~(a.3) and~(b.3) yield $[\si,a_i] = [\si,b_i] = [\si,\z] = [\tsi,\tia_i] = [\tsi,\tib_i] = 1$, which implies that the relations~(a.4),~(a.8) and (b.4) may be removed. The relations~(a.6) become $[a_i,a_j] = [a_i,b_j] = [b_i,b_j] = [b_i,a_j]=1$ for all $i<j$,  the relations~(b.6) become $[\tia_i,\tia_j] = [\tia_i,\tib_j] = [\tib_i,\tib_j] = [\tib_i,\tia_j]=1$ for all $i<j$, the relations~(a.7) become~$[\z,a_i] = [\z,b_i] = 1$, and the relations~(a.5) and~(b.5) become $[a_i,  b_i]=\sigma^2$ and $[\tia_i,  \tib_i]=\tsi^2$  respectively. The relations~(c.1) and~(c.2) reduce to~$[\si,\tia_i] = [\si,\tib_i] =[\tsi,a_i] =[\tsi,b_i] =[\si,\tsi]= 1$, the relations~(c.4) become $[\z,\tia_i] = [\z,\tib_i] = 1$, and for all $i<j$, the relations~(c.5) reduce to $[\tia_{i},a_{j}] = 1$, and the relations~(c.6) become $[\tib_{i},b_{j}] = 1$.  The relations~(c.7.1) and~(c.8.1), which only exist if $g\geq 1,$, yield  $[\tia_i,b_i]= \z$ and $[\tib_i,  a_i] = \z^{-1}$ respectively, the latter being equivalent to $[a_i,\tib_i] = \z$. The other relations of type~(c.7) and~(c.8) reduce to $[\tia_i,b_j] =  [a_i,\tib_j]=1$ for all $i\neq j$, and so we obtain the required presentation.\qedhere
\end{enumerate}
\end{proof}

As a consequence of Proposition~\ref{prop:techresult}, we may exhibit a presentation of $\bsurf{k,n}{g}{}/\Gamma_3(\bsurf{k,n}{g}{})$, which will allow us later to decompose this quotient group as a semi-direct product.

\begin{prop}\label{lem:presgamma}
Let  $k,n\geq 3$ and $g\geq 0$. Let $\rho_{k,n} :\bsurf{k,n}{g}{}\to \bsurf{k,n}{g}{}/\Gamma_3(\bsurf{k,n}{g}{})$ denote the canonical projection. The group $\bsurf{k,n}{g}{}/\Gamma_3(\bsurf{k,n}{g}{})$ admits the following group presentation:\\
\textbf{Generating set:} $\Omega = \rho_{k,n}(\Omega_{k,n})=\{ \si, \tsi, \zeta\}\cup AB \cup\widetilde{AB}$;\\
\textbf{Relations:}
\begin{enumerate}[(a)]
\item $xy = yx$ for all $x,y\in \Omega$, $x\neq y$, and $\{x,y\}\notin \bigl\{\{a_i, b_i\},  \{\tia_j,\tib_j\}, \{b_i,\tia_i\}, \{\tib_i,a_i\} ; i=1,\ldots, g\bigr\}$;
\item $[a_i,  b_i]=\sigma^2$ and $[\tia_i,  \tib_i]=\tsi^2$; $[a_i,  \tib_i] = [\tia_i,b_i]= \zeta \text{ for  all } i=1,
\ldots, g$.
\end{enumerate}
In particular, $\rho_{k,n}(S)= \{\si\}$, $\rho_{k,n}(\widetilde{S})= \{\tsi\}$ and $\rho_{k,n}(Z)= \{\z\}$. 
\end{prop}


\begin{proof}
We apply Proposition~\ref{prop:techresult} with $H = \bsurf{k,n}{g}{}/\Gamma_3(\bsurf{k,n}{g}{})$ and $\rho_H = \rho_{k,n}$. We start by showing that the hypotheses of this proposition are indeed satisfied. By definition the normal subgroup $\Gamma_3(\bsurf{k,n}{g}{})$ is generated by the infinite set of elements of the form $[g_1,[g_2,g_3]]$ where $g_1,g_2,g_3$ range over the elements of $\bsurf{k,n}{g}{}$. We have $[\si_i,[\si_{i+1},\si_i]] = \si_i\si_{i+1}\si_i\si_{i+1}^{-1}\si_i^{-1}\si_{i}^{-1}\si_{i}\si_{i+1}\si_{i}^{-1}\si_{i+1}^{-1} =$ $ \si_{i+1}\si_i\si_{i+1}\si_{i+1}^{-1}\si_i^{-1}\si_{i+1}\si_{i}^{-1}\si_{i+1}^{-1}= \si^2_{i+1}\si_{i}^{-1}\si_{i+1}^{-1}$, for all~$i\in \{1,\ldots,k-2\}$. Since $\rho_{k,n}([\si_i,[\si_{i+1},\si_i]]) = 1$, it follows that $(\rho_{k,n}(\si_{i+1}))^2 = \rho_{k,n}(\si_{i})\rho_{k,n}(\si_{i+1})$, so $\rho_{k,n}(\si_{i+1}) = \rho_{k,n}(\si_{i})$, and hence all of the $\si_i$ have a common image under $\rho_{k,n}$ that we denote by~$\si$. Similarly, all of the $\tsi_i$ have a common image~$\tsi$ under $\rho_{k,n}$. Applying Proposition~\ref{prop:techresult}, the group~$\bsurf{k,n}{g}{}/\Gamma_3(\bsurf{k,n}{g}{})$ admits the presentation given in that proposition. The relations of the form $w=1$ for all $w\in R_{\Omega}$ correspond to the relations $[x,[y,z]] = 1$, where $x,y,z$ range over all the words on~$\Omega\cup\Omega^{-1}$. Comparing this presentation with that given in the statement of the corollary, it suffices to show that these relations are consequences of the  relations~(a) and~(b) given in the statement of the corollary. Let $\equiv$ denote the equivalence relation on words on $\Omega\cup\Omega^{-1}$ generated by the defining relations~(a) and~(b), and
let $x,y,z$ be words on~$\Omega\cup\Omega^{-1}$. We wish to show that $[x,[y,z]]\equiv 1$, where $1$ denotes the empty word. Applying the Witt-Hall commutator identities~\cite[Theorem~5.1(9) and~(10)]{MKS}, induction on word length, and relations~(a) and~(b), we see that $[y,z]$ commutes with all words on~$\Omega\cup\Omega^{-1}$, and so $[x,[y,z]]\equiv 1$ as required.
\end{proof}


\begin{cor}\label{mixedAbelianisation}\mbox{}
Let $k, n \ge 3$. The  group~$B_{k,n}/\Gamma_3(B_{k,n})$ coincides with~$B_{k,n}/\Gamma_2(B_{k,n})$.
\end{cor}

\begin{proof}
Taking $g=0$ in Proposition~\ref{lem:presgamma}, we see that the group~$B_{k,n}/\Gamma_3(B_{k,n})$ is isomorphic to~$\Z^3$, and is therefore Abelian. Thus $\Gamma_2(B_{k,n})$ is a subgroup of~$\Gamma_3(B_{k,n})$, and since the converse inclusion holds by definition, we conclude that~$\Gamma_3(B_{k,n}) = \Gamma_2(B_{k,n})$.
\end{proof}

With respect to Corollary~\ref{mixedAbelianisation}, one may ask whether $\Gamma_3(B_{k,n})=\Gamma_2(B_{k,n})$ if $k\leq 2$ or if $n\leq 2$. If $k=1$ (resp.\ $n=1$), it can be checked easily using Proposition~\ref{prop:presMB} (see also~\cite{GG5}) that the group~$B_{1,n}$ (resp.\ $B_{k,1}$) is isomorphic to the $m$-string braid group $B_{m}(\surf{0}{, 1})$ of the annulus, where $m=n$ (resp.\ $m=k$). If $m=1$ then $B_{m}(\surf{0}{, 1})\cong \Z$, and $\Gamma_2(B_m(\surf{0}{, 1})) = \Gamma_3(B_m(\surf{0}{, 1}))$ trivially. If $m\ge 3$ then $\Gamma_2(B_m(\surf{0}{, 1})) = \Gamma_3(B_m(\surf{0}{, 1}))$~\cite{BGG,GG5}. So suppose that $m=2$. Then $\Gamma_2(B_2(\surf{0}{, 1}))/\Gamma_3(B_2(\surf{0}{, 1}))\cong \Z_{2}$, and in fact $B_2(\surf{0}{, 1})$ is residually nilpotent~\cite{GG5}. This deals with the cases where one of $k$ and $n$ is equal to $1$. The only remaining case is that of $k=n=2$. We do not know whether the natural surjection $B_{2,2}/\Gamma_3(B_{2,2}) \to B_{2,2}/\Gamma_2(B_{2,2})$ is injective (in particular there is no reason for the images of $\zeta_1$ and $\sigma_1$ in $B_{2,2}/\Gamma_3(B_{2,2})$ to commute). Note however that if $\Gamma_2(B_{2,2})/\Gamma_3(B_{2,2})$ were non trivial then it would be isomorphic to the direct sum of a finite (non-zero) number of copies of $\Z_{2}$~\cite{GG5}.

\medskip

If $g\geq 1$, Proposition~\ref{lem:presgamma} allows us to give a more precise description of $\bsurf{k,n}{g}{}/\Gamma_3(\bsurf{k,n}{g}{})$.

\begin{cor}\label{cor:decgammamix}
Let  $k,n\geq 3$, and let $g\geq 1$.
\begin{enumerate}[(i)]
\item The group $\bsurf{k,n}{g}{}/\Gamma_3(\bsurf{k,n}{g}{})$ is isomorphic to a semi-direct product of the form $\left(\Z^3  \times \Z^{2g} \right) \rtimes \Z^{2g}$, the first factor $\Z^3$ being generated by $\{\sigma, \tsi, \zeta\}$, the second factor $\Z^{2g}$
by $\{ a_1, \ldots, a_g, \tia_1, \ldots,  \tia_g \}$, and the third factor $\Z^{2g}$ by 
$\{ b_1, \ldots, b_g, \tib_1, \ldots,  \tib_g \}$.
\item Any element~$\gamma$ of $\bsurf{k,n}{g}{}/\Gamma_3(\bsurf{k,n}{g}{})$ may be written uniquely in the form:
\begin{equation}\label{eq:gamma}
\gamma = \sigma^p \tsi^q \zeta^r  \prod_{i=1}^g a_i^{m_i}  \tia_i^{\mskip 4mu\widetilde{m}_i} \prod_{i=1}^g b_i^{n_i}  \tib_i^{\mskip 4mu\widetilde{n}_i}, \quad\text{where $p,q,r,m_{i},\widetilde{m}_i,n_{i},\widetilde{n}_i\in \Z$}.
\end{equation}
\item The centre of the group $\bsurf{k,n}{g}{}/\Gamma_3(\bsurf{k,n}{g}{})$ is  the free Abelian group of rank three generated by~$\si,\tsi$ and $\z$. 
\end{enumerate}
\end{cor}

\begin{proof}\mbox{}
\begin{enumerate}[(i)]
\item 
%
For a set $X$, let $F\!A(X)$ denote the free Abelian group on $X$. Let $U= \{\sigma, \tsi, \zeta, a_1, \ldots, a_g, \tia_1, \ldots,  \tia_g \}$ and let $V=\{b_1, \ldots, b_g, \tib_1, \ldots, \tib_g \}$. One may check that the map $\phi : V\to \operatorname{Aut}(F\!A(U))$ defined by:
\begin{equation*}
\phi(v)(u)=
\begin{cases}
\si^{-2}u & \text{if $v=b_{i}$ and $u=a_{i}$, or if $v=\tib_{i}$ and $u=\tia_{i}$}\\
\z^{-1}u & \text{if $v=b_{i}$ and $u=\tia_{i}$, or if $v=\tib_{i}$ and $u=a_{i}$}\\
u & \text{otherwise,}
\end{cases}
\end{equation*}
is well defined, and that it extends to a homomorphism $\phi : F\!A(V)\to \operatorname{Aut}(F\!A(U))$. We may thus form the semi-direct product $F\!A(U) \rtimes_{\phi} F\!A(V)$, and by standard results (\cite{joh} for example), this group admits a presentation that coincides with that of $\bsurf{k,n}{g}{}/\Gamma_3(\bsurf{k,n}{g}{})$ given in Proposition~\ref{lem:presgamma}.
 
\item follows directly from~(i).

\item The centrality of $\si, \tsi$ and $\z$ follows from Proposition~\ref{lem:presgamma}, and the fact that the subgroup $\langle \si, \tsi,\z\rangle$ is free Abelian of rank three follows from~(i). To see that this subgroup is indeed the centre of $\bsurf{k,n}{g}{}/\Gamma_3(\bsurf{k,n}{g}{})$, let $\gamma$ be as in equation~(\ref{eq:gamma}), and suppose that some $m_j$, $\widetilde{m}_j$, $n_j$ or $\widetilde{n}_j$ is non zero. If $m_j\neq 0$ say, then $b_j \gamma  b_j^{-1}= (\sigma^{-2m_j})\z^{-\widetilde{m}_j}\gamma$, so $[\gamma,b_j] = \sigma^{-2m_j}\z^{-\widetilde{m}_j}\neq 1$ since the set $\{\si,\z\}$ generates a free Abelian group of  rank~$2$, and thus $\gamma$ is non central. By replacing $b_j$ by $a_j$, $\tib_j$ and $\tia_j$ respectively, we obtain the same conclusion in the cases $n_j,\widetilde{m}_j, \widetilde{n}_j\neq 0$. So if $\gamma$ is central then $m_j = n_j = \widetilde{m}_j =  \widetilde{n}_j =  0$ for all $j=1,\ldots,g$, and the result follows.\qedhere
\end{enumerate}
\end{proof}

Before going any further, we state and prove the following lemma that shall be used at various points in the rest of the paper. If $H$ is a group then we denote its centre by $Z(H)$.
\begin{lem}\label{lem:equivinjective}
Let $K,H$ and $\widetilde{H}$ be groups, let $\gamma:K \to H$ and $\tau:H \to \widetilde{H}$ be homomorphisms, and let $\widetilde{\tau}: K \to \widetilde{H}$ be defined by $\widetilde{\tau}=\tau\circ \gamma$. Assume that $\Gamma_{2}(H)\subset \Im{\gamma}$ and that $Z(H)\subset \Im{\gamma}$. If $\widetilde{\tau}$ is injective then $\tau$ is injective. In particular, if $\gamma$ is injective, then $\widetilde{\tau}$ is injective if and only if $\tau$ is injective. 
\end{lem}

\begin{proof}
Suppose that $\Gamma_{2}(H)\subset \Im{\gamma}$, $Z(H)\subset \Im{\gamma}$ and that $\widetilde{\tau}$ is injective, and let us prove that $\tau$ is injective. Let $g\in \ker{\tau}$. If $g'\in H$ then $[g,g']\in \Gamma_{2}(H)$, so there exists $k\in K$ such that $\gamma(k)=[g,g']$. Then $\widetilde{\tau}(k)= \tau\circ \gamma(k)=[\tau(g),\tau(g')]=1$ since $g\in \ker{\tau}$. Thus $k=1$ by injectivity of $\widetilde{\tau}$, and hence $[g,g']=1$ for all $g'\in H$. It follows that $g\in Z(H)$, so there exists $k'\in K$ such that $\gamma(k')=g$. As in the previous sentence, we conclude that $k'=1$, and so $g=1$, which proves the injectivity of $\tau$. The second assertion then follows easily.
\end{proof}

The following result  says that the 
inclusion of classical mixed braid groups in surface mixed braid groups induces an embedding on the level of metabelian quotients. As we saw in Corollary~\ref{nolinearextK1}, there is no such embedding on the level of Abelian quotients.

\begin{cor} \label{cor:themain2iii}
Let $k,n\geq 3$, and let $g\geq 1$. Let $H$ be a group, and let $\rho_H : B_{k,n}(\surf{g}{}) \to H$ be a homomorphism.
\begin{enumerate}[(i)] 
\item The following conditions are equivalent:
\begin{enumerate}[(a)]
\item the homomorphism $\rho_H$ factors through~$\rho_{k,n}: B_{k,n}(\surf{g}{})\to B_{k,n}(\surf{g}{})/\Gamma_3(B_{k,n}(\surf{g}{}))$.
\item the restriction of $\rho_H$ to $B_{k,n}$ factors through~$r_{k,n}: B_{k,n}\to B_{k,n}/\Gamma_2(B_{k,n})$.
\item there exist $\si,\tsi\in H$ such that $\rho_{H}(S)= \{\si\}$ and $\rho_{H}(\widetilde{S})= \{\tsi\}$.

\end{enumerate}
\item There exists an injective homomorphism~$\alpha_{k,n}:B_{k,n}/ \Gamma_2(B_{k,n})\to B_{k,n}(\surf{g}{})/ \Gamma_3(B_{k,n}(\surf{g}{}))$ for which $\alpha_{k,n}\circ r_{k,n} = \rho_{k,n}\circ\iota_{k,n}$ and whose image is the centre of $B_{k,n}(\surf{g}{})/ \Gamma_3(B_{k,n}(\surf{g}{}))$.
\item Let $\rho_{H,3}: B_{k,n}(\surf{g}{})/\Gamma_3(B_{k,n}(\surf{g}{}))\to H$ be a homomorphism, and let $\rho_{H,2} : B_{k,n}/\Gamma_2(B_{k,n})\to H$ be the homomorphism defined by $\rho_{H,2}=\rho_{H,3}\circ \alpha_{k,n}$. Then $\rho_{H,3}$ is injective if and only if $\rho_{H,2}$ is injective.
\item 
Let $\rho_{H,2}: B_{k,n}/\Gamma_2(B_{k,n})\to H$ be a homomorphism such that $\rho_{H} \circ \iota_{k,n}= \rho_{H,2} \circ r_{k,n}$. Then there exists a homomorphism $\rho_{H,3}: B_{k,n}(\surf{g}{})/\Gamma_3(B_{k,n}(\surf{g}{}))\to H$ such that $\rho_{H,3}\circ \rho_{k,n}=\rho_{H}$ and $\rho_{H,2}=\rho_{H,3}\circ \alpha_{k,n}$. Furthermore, $\rho_{H,3}$ is injective if and only if $\rho_{H,2}$ is injective. In particular, if $\rho_{H,2}$ is injective and $\rho_{H}$ is surjective then $\rho_{H,3}$ is an isomorphism.
\end{enumerate}
\end{cor}

Note that Theorem~\ref{main_thm_2}(iii) follows by taking $\theta_{H}= \rho_{H,3}$ in Corollary~\ref{cor:themain2iii}(iv). 



\pagebreak
 
\begin{proof}[Proof of Corollary~\ref{cor:themain2iii}] Replacing $H$ by $\rho_H(B_{k,n}(\surf{g}{}))$ if necessary, we may suppose that $\rho_H$ is surjective.
\begin{enumerate}[(i)]
\item The equivalence of~(b) and~(c) follows from Propositions~\ref{lem:presgammaabel} and~\ref{prop:techresult} since $B_{k,n}$ is generated by $S\cup\widetilde{S}\cup Z$.
%
%
The implication (a)$\Longrightarrow$(c) also follows easily from Proposition~\ref{lem:presgamma}. So it suffices to prove the implication (c)$\Longrightarrow$(a). If~(c) holds then Proposition~\ref{prop:techresult} applies, and the comparison of the presentation of $H$ given there with that of $B_{k,n}(\surf{g}{})/\Gamma_3(B_{k,n}(\surf{g}{}))$ given in Proposition~\ref{lem:presgamma} implies that $H$ is a quotient of $B_{k,n}(\surf{g}{})/\Gamma_3(B_{k,n}(\surf{g}{}))$. Thus $\rho_{H}$ factors through $\rho_{k,n}$ as required.

\item Take $H =  B_{k,n}(\surf{g}{})/\Gamma_3(B_{k,n}(\surf{g}{}))$ and $\rho_H = \rho_{k,n}$ in part~(i). Using Proposition~\ref{prop:techresult}, condition~(a) is satisfied, and so applying condition~(b), we deduce the existence of a homomorphism~$\alpha_{k,n}:B_{k,n}/ \Gamma_2(B_{k,n})\to B_{k,n}(\surf{g}{})/ \Gamma_3(B_{k,n}(\surf{g}{}))$ that provides a factorisation of $\rho_{H}$ through $r_{k,n}$, so that $\alpha_{k,n}  \circ r_{k,n}=\rho_{k,n}\circ\iota_{k,n}$. This relation, the surjectivity of $r_{k,n}$ and Proposition~\ref{prop:techresult} imply that the image of $\alpha_{k,n}$ is generated by $\brak{\si, \tsi,\z}$. By Corollary~\ref{cor:decgammamix}(iii), this image is a free Abelian group of rank $3$ with basis $\{ \si,\tsi,\z\}$, and is equal to the centre of $B_{k,n}(\surf{g}{})/ \Gamma_3(B_{k,n}(\surf{g}{}))$. By Proposition~\ref{lem:presgammaabel}, $B_{k,n}/ \Gamma_2(B_{k,n})$ is also a free Abelian group of rank $3$  with basis $\{ \si,\tsi,\z\}$. Since $\alpha_{k,n}$ sends this basis onto the given basis of the image of $\alpha_{k,n}$, we conclude that $\alpha_{k,n}$ is injective.

\item 
We have $\rho_{H,2} = \rho_{H,3}\circ\alpha_{k,n}$, where $\alpha_{k,n}$ is injective. By standard commutator properties, we have $\Gamma_{2}(B_{k,n}(\surf{g}{})/\Gamma_3(B_{k,n}(\surf{g}{})))\subset Z(B_{k,n}(\surf{g}{})/\Gamma_3(B_{k,n}(\surf{g}{})))$. Further, $Z(B_{k,n}(\surf{g}{})/\Gamma_3(B_{k,n}(\surf{g}{})))$ is equal to the image of $\alpha_{k,n}$ by part~(ii). The result then follows by applying Lemma~\ref{lem:equivinjective}.


\item Since the restriction of $\rho_{H}$ to $B_{k,n}$ factors through $r_{k,n}$, condition~(i)(b) is satisfied, and so by condition~(i)(a), there exists a homomorphism $\rho_{H,3}: B_{k,n}(\surf{g}{})/\Gamma_3(B_{k,n}(\surf{g}{}))\to H$ such that $\rho_{H,3}\circ \rho_{k,n}=\rho_{H}$. Using the homomorphism $\alpha_{k,n}$ of part~(ii), we obtain the following diagram:
\begin{equation*}
\begin{tikzcd}
B_{k,n} \arrow{r}{r_{k,n}} \arrow{d}{\iota_{k,n}} &  B_{k,n}/\Gamma_2(B_{k,n}) \arrow[swap]{d}{\rho_{H,2}} \arrow[bend left=40]{rddd}{\alpha_{k,n}} &\\
B_{k,n}(\surf{g}{}) \arrow{r}{\rho_{H}} \arrow[swap,bend right=20]{ddrr}{\rho_{k,n}} &  H &\\
&& \\
&& B_{k,n}(\surf{g}{})/\Gamma_3(B_{k,n}(\surf{g}{})), \arrow[dashed]{uul}{\rho_{H,3}}
\end{tikzcd}
\end{equation*}
which is commutative, except possibly for the relation $\rho_{H,3}\circ \alpha_{k,n}=\rho_{H,2}$. From the commutativity of the rest of the diagram, we see that $\rho_{H,3}\circ \alpha_{k,n}\circ r_{k,n}= \rho_{H,3}\circ \rho_{k,n}\circ \iota_{k,n}= \rho_{H}\circ \iota_{k,n}= \rho_{H,2}\circ r_{k,n}$. The surjectivity of $r_{k,n}$ implies that $\rho_{H,3}\circ \alpha_{k,n}=\rho_{H,2}$, which proves the existence of $\rho_{H,3}$ satisfying the given properties. The equivalence of the injectivity of $\rho_{H,3}$ and that of $\rho_{H,2}$ is given by part~(iii), and the last part then follows easily.\qedhere         
\end{enumerate}
\end{proof}


\subsection{Metabelian quotients of $\bsurf{k}{g,n}{}$ and $\bsurf{n}{g}{}$} \label{sec:metaquotients2} 



Starting with the presentations of Proposition~\ref{presconnues} instead of that of~$B_{k,n}(\surf{g}{})$ given in~Proposition~\ref{prop:presMB}, many of the arguments of Section~\ref{sec:metaquotients} may be repeated for $\bsurf{k}{g,n}{}$ and $\bsurf{n}{g}{}$. As we already saw in Section~\ref{sec:presquotients}, there are some minor differences in some of the proofs, for example, the $\z_i$ are not identified in quotients of $\bsurf{k}{g,n}{}$, and the case $k=1$ gives rise to slightly different results. In what follows, $\bsurf{k}{g}{}$ will also be denoted by $\bsurf{k}{g,0}{}$, and we shall consider its presentation given by Proposition~\ref{presconnues}(i) with $Z=\varnothing$ subject to the relations (a.1)--(a.6). We emphasise that in this section, we adopt the convention that if $g=0$ (resp.\ $n=0$, $k=1$) then $AB$ (resp.\ $Z$, $\brak{\si}$) should be suppressed from the list of generators, and that any relations involving its elements should also be removed. Propositions~\ref{prop:techresult3} and~\ref{lem:presbng} and Corollaries~\ref{cor:decgamma} and~\ref{themain2iv} are the analogues for $\bsurf{k}{g,n}{}$ of Propositions~\ref{prop:techresult} and~\ref{lem:presgamma} and Corollaries~\ref{cor:decgammamix} and~\ref{cor:themain2iii} respectively, and their proofs, which we leave to the reader, are similar.


\begin{prop} \label{prop:techresult3}
Let $k\geq 3$, let $g,n\geq 0$, let $G$ be a group, and let $\rho_G : \bsurf{k}{g,n}{} \to G$ be a surjective homomorphism. Consider the presentation of $\bsurf{k}{g,n}{}$ given by Proposition~\ref{presconnues}, and let $R$ be a set of words on~$(S\cup AB \cap Z)\cup (S\cup AB\cup Z)^{-1}$ whose normal closure in $\bsurf{k}{g,n}{}$ generates $\ker{(\rho_{G})}$.  Suppose that there exists $\si\in G$ such that $\rho_{G}(S)= \{\si\}$.  Let $R_{G,\si}$ be the set of words obtained from $R$ by replacing all of the $\si_i$ by $\si$ (and $\si_i^{-1}$ by $\si^{-1}$). Then the group~$G$ has the following presentation:\\ 
\textbf{Generating set:} $\rho_{G}(S \cup AB \cup Z)=\{\si\}\cup AB \cup Z$;\\
\textbf{Relations:}
\begin{enumerate}[\textbullet]
\item $w = 1$ for all $w\in R_{G,\si}$;
\item $xy = yx$ for all $x,y\in \{\si\}\cup AB \cup Z$, $x\neq y$, and $\{x,y\}\notin \bigl\{\{a_i,b_i\} ; i=1, \ldots, g\bigr\}$;
\item $[a_i,b_i]=\si^2$ for all $i=1,\ldots, g$.
\end{enumerate}
\end{prop}



Let $\rho_{k}: \bsurf{k}{g,n}{}\to \bsurf{k}{g,n}{}/\Gamma_3(\bsurf{k}{g,n}{})$ denote the canonical projection. Note that if $n=0$ then this is the homomorphism $\rho_{k}$ defined in the Introduction.

\begin{prop}\label{lem:presbng} 
Let $k\geq 3$, and let $g,n\geq 0$. The group~$\bsurf{k}{g,n}{}/\Gamma_3(\bsurf{k}{g,n}{})$ admits the following group presentation:\\
\noindent \textbf{Generating set:} $\rho_{k}(S\cup AB\cup Z)=\{\si\}\cup AB \cup Z$;\\
\noindent  \textbf{Relations:}
\begin{enumerate}[\textbullet]
\item $xy = yx$ for all $x,y\in \{\si\}\cup AB \cup Z$, $x\neq y$, and $\{x,y\}\notin \bigl\{\{a_i,b_i\} ; i=1, \ldots, g\bigr\}$;
\item for all $i=1,\ldots, g$, $[a_i,b_i]=\si^2$.
\end{enumerate}
\end{prop}

As in Corollary~\ref{mixedAbelianisation}, we deduce from Proposition~\ref{lem:presbng} that $B_{k}(\D_{n})/\Gamma_3(B_{k}(\D_{n}))$ coincides with its Abelianisation $B_{k}(\D_{n})/\Gamma_2(B_{k}(\D_{n}))$.

\begin{cor}\label{cor:decgamma}
Let $k\geq 3$, $g\geq 1$ and $n\geq 0$.
\begin{enumerate}[(i)]
\item The group~$\bsurf{k}{g,n}{}/\Gamma_3(\bsurf{k}{g,n}{})$ is isomorphic to a semi-direct product of the form
$\left( \Z^{n+1} \times \Z^g \right) \rtimes \Z^g$.
More precisely, the first factor~$\Z^{n+1}$ is generated by $\{\si\}\cup Z$, the second factor $\Z^g$
is generated by $\{a_1, \ldots, a_g\}$, and the third factor $\Z^g$ is generated by $\{b_1, \ldots, b_g\}$.

\item Every element~$\gamma \in \bsurf{k}{g,n}{}/\Gamma_3(\bsurf{k}{g,n}{})$ may be written uniquely in the form:
\begin{equation*}
\text{$\gamma= \sigma^p \prod_{i=1}^n \z_i^{q_i} \prod_{i=1}^g a_i^{m_i} \prod_{i=1}^g b_i^{n_i}$, where $p, q_{i},m_{i},n_{i}\in \Z$.}
\end{equation*}

\item The centre of the group $\bsurf{k}{g,n}{}/\Gamma_3(\bsurf{k}{g,n}{})$ is isomorphic to~$\Z^{n+1}$ and is generated by $\{\si\}\cup Z$.
\end{enumerate}
\end{cor}



Part of the following result contrasts with Corollary~\ref{nolinearextK1}.  More precisely the 
inclusion of classical braid groups in surface braid groups induces an embedding at the level of metabelian quotients, but as we saw in Corollary~\ref{nolinearextK1},  there is no such embedding at the level of Abelian quotients.

\begin{cor}\label{themain2iv}
Let $k\geq 3$, $g\geq 1$ and $n\geq 0$. Let $G$ be a group, and let $\rho_G : \bsurf{k}{g,n}{} \to G$ be a homomorphism.
\begin{enumerate}[(i)]
\item the following conditions are equivalent:
\begin{enumerate}[(a)]
\item the homomorphism $\rho_G$ factors through~$\rho_{k}: \bsurf{k}{g,n}{}\to \bsurf{k}{g,n}{}/\Gamma_3(\bsurf{k}{g,n}{})$.
\item the restriction of $\rho_G$ to $B_{k}(\D_n)$ factors through~$r_{k}: B_{k}(\D_n)\to B_{k}(\D_n)/\Gamma_2(B_{k}(\D_n))$.
\item There exists $\si \in G$ such that $\rho_{G}(S)= \{\si\}$.
\end{enumerate}

\item There exists an injective homomorphism~$\alpha_{k}:B_{k}(\D_n)/ \Gamma_2(B_{k}(\D_n))\to \bsurf{k}{g,n}{}/ \Gamma_3(\bsurf{k}{g,n}{})$ that satisifies $\alpha_{k}\circ r_k = \rho_{k}\circ\iota_{k}$ and whose image is the centre of $B_{k}(\surf{g,n}{})/ \Gamma_3(\bsurf{k}{g,n}{})$.

\item Let $\rho_{G,3}: \bsurf{k}{g,n}{}/\Gamma_3(\bsurf{k}{g,n}{})\to G$ be a homomorphism, and consider the homomorphism $\rho_{G,2} : B_{k}(\D_n)/\Gamma_2(B_{k}(\D_n))\to G$ defined by $\rho_{G,2}=\rho_{G,3}\circ \alpha_{k}$. Then $\rho_{G,3}$ is injective if and only if $\rho_{G,2}$ is.

\item 
Let $\rho_{G,2}: B_{k}(\D_n)/ \Gamma_2(B_{k}(\D_n))\to G$ be a homomorphism such that $\rho_{G} \circ \iota_{k}= \rho_{G,2} \circ r_{k}$. Then there exists a homomorphism $\rho_{G,3}: \bsurf{k}{g,n}{}/\Gamma_3(\bsurf{k}{g,n}{})\to G$ such that $\rho_{G,3}\circ \rho_{k}=\rho_{G}$ and $\rho_{G,2}=\rho_{G,3}\circ \alpha_{k}$. Furthermore, $\rho_{G,3}$ is injective if and only if $\rho_{G,2}$ is injective. In particular, if $\rho_{G,2}$ is injective and $\rho_{G}$ is surjective then $\rho_{G,3}$ is an isomorphism.
\end{enumerate}
\end{cor}

Corollary~\ref{themain2iv} is may be proved in the same way as Corollary~\ref{cor:themain2iii}: in the proof, Propositions~\ref{lem:presgammaabel},~\ref{prop:techresult} and~\ref{lem:presgamma} should be replaced by Propositions~\ref{lem:presbngabel},~\ref{prop:techresult3} and~\ref{lem:presbng} respectively.

We are now able to complete the proof of parts~(i),~(iii) and~(iv) of Theorem~\ref{main_thm_2}. Part~(ii) will be proved in Section~\ref{section4}.

\begin{proof}[Proof of parts~(i),~(iii) and~(iv) of Theorem~\ref{main_thm_2}]
Part~(iii) was proved just after the statement of Corollary~\ref{cor:themain2iii}, and part~(iv) follows in a similar way by Corollary~\ref{themain2iv}(iv). 
It remains to prove part~(i). Let $\alpha_{k,n}:B_{k,n}/ \Gamma_2(B_{k,n})\to B_{k,n}(\surf{g}{})/ \Gamma_3(B_{k,n}(\surf{g}{}))$ and $\alpha_{n}:B_{n}/ \Gamma_2(B_{n})\to \bsurf{n}{g}{}/ \Gamma_3(\bsurf{n}{g}{})$ be the homomorphisms defined in Corollaries~\ref{cor:themain2iii}(ii) and~\ref{themain2iv}(ii) respectively. Together with the commutative diagrams~(\ref{eq:diagartin}),~(\ref{diagramexactsequencea}) and~(\ref{eq:injnat}), these  two corollaries entail the existence and the commutativity of the diagram~(\ref{diagramexactsequencethm}), with the exception, \emph{a priori}, of the existence of $\gamma_k$ and the commutativity of the first column, which we now prove. 
Let $\gamma_k: G_k\to  G_k\bigl(\surf{g}{}\bigr)$ denote the restriction of $\alpha_{k,n}$ to $G_{k}$. The commutativity of the rest of the diagram~(\ref{diagramexactsequencethm}) implies that $\gamma_{k}$ is well defined, and the injectivity of $\gamma_{k}$ is a consequence of that of $\alpha_{k,n}$. Restricting appropriately the relation $\alpha_{k,n}\circ r_{k,n} = \rho_{k,n}\circ\iota_{k,n}$, we obtain $\gamma_k\circ p_k = \Phi_k\circ\iota_k$, and this completes the proof of Theorem~\ref{main_thm_2}(i).
\end{proof}

\section{The group $G_k\bigl(\surf{g}{}\bigr)$}\label{section4}

 
In this  section, we  exhibit a presentation of $G_k\bigl(\surf{g}{}\bigr)$, we prove Theorem~\ref{main_thm} and   we complete the proof of Theorem~\ref{main_thm_2}.  


Let $k,n\geq 1$ (resp.\ $k,n\geq 3$). The group $G_{k}$ (resp.\ $G_k\bigl(\surf{g}{}\bigr)$), which is defined by the commutative diagram~(\ref{eq:diagartin}) (resp.~(\ref{diagramexactsequencea})), is described in Lemma~\ref{lemGkpres} (resp.\ Proposition~\ref{fondhk}). Notice in particular that these groups only depend on $k$ and $g$, and do not depend on $n$, which justifies the absence of $n$ in the notation.

\begin{lem} \label{lemGkpres}
Let $k,n\geq 1$. Then the group $G_k$ is a free Abelian group and is a direct factor of~$B_{k,n}/\Gamma_2(B_{k,n})$. If $k = 1$ then $G_k$ is isomorphic to $\Z$ and is generated by $\{\z\}$; if $k\geq 2$ then $G_k$ is of rank~$2$, and is generated by $\{\si,\z\}$. 
\end{lem}

\begin{proof} 
Let $k,n\geq 1$. We make use of the notation of Propositions~\ref{presconnues} and~\ref{prop:presMB} for the groups $B_{k,n}$ and $B_{n}$. Applying Propositions~\ref{lem:presgammaabel} and~\ref{lem:presbngabel}, the group~$B_{k,n}/\Gamma_2(B_{k,n})$ is a free Abelian group with basis $\widehat{S} \cup \{\z\}$, and by Proposition~\ref{lem:presbngabel}, $B_n/\Gamma_2(B_n)$ is a free Abelian group with basis $\widehat{S} \setminus \{ \si\}$, where $\widehat{S}$ is as defined in Proposition~\ref{lem:presgammaabel}. 

In terms of the notation of Proposition~\ref{lem:presgammaabel}, the group~$B_{k,n}/\Gamma_2(B_{k,n})$ is a free Abelian group with basis $\widehat{S} \cup \{\z\}$, and by Proposition~\ref{lem:presbngabel}, $B_n/\Gamma_2(B_n)$ is a free Abelian group with basis $\widehat{S} \setminus \{ \si\}$. The homomorphism from $B_{k,n}$ to $B_n$ sends $\tsi_1,\ldots,\tsi_{n-1}$ to $\tsi_1,\ldots,\tsi_{n-1}$ respectively  (if $n\geq 2$), and $\si_1,\ldots,\si_{k-1}$ (if $k\geq 2$) and $\z_1,\ldots, \z_n$ onto the trivial element. Since $r_{k,n}$ identifies the elements of $S$ (resp.\ of $\widetilde{S}$, of $Z$) to $\brak{\si}$ (resp.\ to $\brak{\tsi}$, to $\brak{\z}$), it follows that $G_k$ is the kernel of the homomorphism from~$B_{k,n}/\Gamma_2(B_{k,n})$ to $B_n/\Gamma_2(B_n)$ that sends $\tsi$ to $\tsi$ (if $n\geq 2$), and $\si$ (if $k\geq 2$) and $\z$ onto the trivial element. Hence $G_{k}$ is the free Abelian group with basis $\bigl(\widehat{S}\setminus \brak{\tsi}\bigr)\cup \brak{\z}$. In particular, $G_{k}\cong \Z$ if $k=1$, and $G_{k}\cong \Z^{2}$ if $k\geq 2$.
\end{proof}

\begin{prop} \label{fondhk}
Let $k,n \geq 3$ and $g\geq 1$. The group $G_k\bigl(\surf{g}{}\bigr)$ admits the following presentation:\\
\noindent\textbf{Generating set:} $\{\sigma, \zeta\}\cup AB$;\\
\textbf{Relations:}
\begin{enumerate}[\textbullet]
\item $xy=yx$ for all $x,y\in \{\si, \z\}\cup AB$, $x\neq y$, and $\{x,y\}\notin \bigl\{\{a_i,b_i\} ; i=1, \ldots, g\bigr\}$;
\item $[a_i,b_i] = \sigma^2$ for all $i=1,\ldots,g$.
\end{enumerate}
\end{prop}

\begin{proof}
From the commutative diagram~(\ref{diagramexactsequencea}) of short exact sequences, the homomorphism $\Phi_{k}: B_{k}(\surf{g,n}{})\to G_k\bigl(\surf{g}{}\bigr)$ is the restriction of $\rho_{k,n}$ to $B_{k}(\surf{g,n}{})$. 
In terms of the presentations of Propositions~\ref{presconnues} and~\ref{prop:presMB}, $\psi_{k}$ is defined by $\psi_{k}(x) = x$ for all $x\in \widetilde{S}\cup\widetilde{AB}$, and $\psi_{k}(x) = 1$ for all $x\in S\cup AB\cup Z$. It follows from Proposition~\ref{lem:presgamma} and the commutativity of the diagram~(\ref{diagramexactsequencea}) that the subgroup of $B_{k,n}(\surf{g}{})/ \Gamma_3(B_{k,n}(\surf{g}{}))$ 
generated by  $\{\si, \zeta\}\cup AB$ is included in~$G_k\bigl(\surf{g}{}\bigr)$. Conversely, let $\gamma\in B_{k,n}(\surf{g}{})/ \Gamma_3(B_{k,n}(\surf{g}{}))$, that we write in the form of equation~(\ref{eq:gamma}). Since $\overline{\psi}_{k}$ is induced by $\psi_{k}$, we see that $\overline{\psi}_{k}(\gamma)=\tsi^q \prod_{i=1}^g \tia_i^{\mskip 4mu\widetilde{m}_i} \prod_{i=1}^g \tib_i^{\mskip 4mu\widetilde{n}_i}$. Corollary~\ref{cor:decgamma}(ii) implies that $\gamma\in \ker{(\overline{\psi}_{k})}$ if and only if $q = \widetilde{m}_i = \widetilde{n}_i = 0$ for all $i=1,\ldots,g$, and so $\gamma\in \ang{\{\si, \zeta\}\cup AB}$ by equation~(\ref{eq:gamma}). Consequently, $G_k\bigl(\surf{g}{}\bigr)$ is the subgroup of $B_{k,n}(\surf{g}{})/\Gamma_3(B_{k,n}(\surf{g}{}))$ generated by $\{\si, \zeta\}\cup AB$. Since the relations given in the statement of the proposition hold in~$B_{k,n}(\surf{g}{})/ \Gamma_3(B_{k,n}(\surf{g}{}))$, they also hold in~$G_k\bigl(\surf{g}{}\bigr)$. Moreover, starting from any word~$w$ on the elements of $\{\sigma, \zeta\}\cup AB$ and their inverses, it follows from just these relations (and the relations $x^{-1}x=xx^{-1} = 1$ for $x\in \{\sigma, \zeta\}\cup AB$) that $w$ may be transformed into the unique word of the form $\sigma^p \zeta^r  \prod_{i=1}^g a_i^{m_i}\prod_{i=1}^g b_i^{n_i}$, $p,r,m_{i},n_{i}\in \Z$,  that represents the same element in~$B_{k,n}(\surf{g}{})/ \Gamma_3(B_{k,n}(\surf{g}{}))$. This proves that the relations given in the proposition are indeed a complete set of relations for~$G_k\bigl(\surf{g}{}\bigr)$ for the generating set $\{\sigma, \zeta\}\cup AB$.
\end{proof}

Let $k,n\geq 3$ and $g\geq 1$. The homomorphism~$\gamma_{k}:G_{k}\to G_k\bigl(\surf{g}{}\bigr)$, which is the restriction of $\alpha_{k,n}$ to $G_{k}$, was seen to be injective in Theorem~\ref{main_thm_2}(i).

\begin{cor}\label{cor:decgamma4}
Let $g\geq 1$.
\begin{enumerate}[(i)]
\item Let $k\geq 1$ and $n\geq 1$. Then the homomorphism~$\Phi_{k}: B_{k}(\surf{g,n}{})\to G_k\bigl(\surf{g}{}\bigr)$ factors through the homomorphism $\rho_{k}: \bsurf{k}{g,n}{}\to \bsurf{k}{g,n}{}/\Gamma_3(\bsurf{k}{g,n}{})$. 

\item Let $k,n\geq 3$. The group~$G_k\bigl(\surf{g}{}\bigr)$ is isomorphic to a semi-direct product of the form~$(\mathbb{Z}^2\times\mathbb{Z}^{g})\rtimes \mathbb{Z}^{g}$. Its centre is a free Abelian group with basis~$\{\si,\z\}$ and is equal to $\gamma_{k}(G_{k})$. Moreover, every element of $G_k\bigl(\surf{g}{}\bigr)$ may be written uniquely in the form $\si^p\z^q\prod_{i=1}^g a_i^{m_i}\prod_{i=1}^g b_i^{n_i}$, where $p,q,m_{i},n_{i}\in \Z$.

\item Let $k,n\geq 3$. The groups~$G_k\bigl(\surf{g}{}\bigr)$ and $M_k\bigl(\surf{g}{}\bigr) := \bigl(\bsurf{k}{g,n}{}/\Gamma_3(\bsurf{k}{g,n}{})\bigr)\bigl/_{\z_1= \cdots = \z_n}\bigr.$ are isomorphic. 
Moreover, if $q_k: \bsurf{k}{g,n}{}/\Gamma_3(\bsurf{k}{g,n}{}) \to G_k\bigl(\surf{g}{}\bigr)$ is a homomorphism for which $\Phi_k = q_k\circ \rho_k$ then $q_k$ induces an isomorphism from  $M_k\bigl(\surf{g}{}\bigr) $ to~$G_k\bigl(\surf{g}{}\bigr)$.
\end{enumerate}
\end{cor}

\begin{proof}\mbox{}
\begin{enumerate}[(i)]
\item From the commutative diagram~(\ref{diagramexactsequencea}) of short exact sequences, $G_k\bigl(\surf{g}{}\bigr)$ is a subgroup of the quotient $\bsurf{k,n}{g}{}/\Gamma_3(\bsurf{k,n}{g}{})$, so $[x,[y,z]] = 1$ for all $x,y,z\in G_k\bigl(\surf{g}{}\bigr)$, and hence $\Phi_k$ factors through $\rho_{k}: \bsurf{k}{g,n}{}\to \bsurf{k}{g,n}{}/\Gamma_3(\bsurf{k}{g,n}{})$.

\item The result follows from Proposition~\ref{fondhk} and Lemma~\ref{lemGkpres} using arguments similar to those given in the proof of Corollary~\ref{cor:decgammamix}.

\item This is a consequence of the presentations of the groups~$\bsurf{k}{g,n}{}/\Gamma_3(\bsurf{k}{g,n}{})$ and $G_k\bigl(\surf{g}{}\bigr)$ given in Propositions~\ref{lem:presbng} and~\ref{fondhk} respectively.\qedhere
\end{enumerate}
\end{proof}

We are now in a position to prove Theorem~\ref{main_thm} and to finish the proof of Theorem~\ref{main_thm_2}. 
\begin{proof}[Proof of Theorem~\ref{main_thm}]
The proof is a straightforward consequence of Proposition~\ref{fondhk}. Indeed, for $k\ge 3$, the group $G_\Sigma$ introduced in~\cite[Section~3]{HK} was abstractly defined by the group presentation of Proposition~\ref{fondhk}. Hence there exists a canonical isomorphism of groups~$\iota: G_\Sigma\to G_k\bigl(\surf{g}{}\bigr)$. Moreover, the homomorphism $\Phi_k: B_k(\surf{g}{, n})\to G_k\bigl(\surf{g}{}\bigr)$ is the restriction of $\rho_{k,n}: B_{k,n}(\surf{g}{})\to B_{k,n}(\surf{g}{})/ \Gamma_3(B_{k,n}(\surf{g}{}))$, and by Proposition~\ref{lem:presgamma}, is defined by $\Phi_k(S) = \brak{\si}$, $\Phi_k(Z) = \brak{\z}$,  and $\Phi_k(x) = x$ for all $x\in AB$. So considering the definition of the homomorphism~$\Phi_\Sigma$ given in~\cite[page~266]{HK}, we conclude that~$\iota\circ\Phi_\Sigma = \Phi_k$ as required.
\end{proof}

\begin{cor}\label{mainthipr2} Let  $k,n\geq 3$ and $g\geq 1$. Let $G$ be a group, and $\Phi_G : \bsurf{k}{g,n}{} \to G$ be a homomorphism.
\begin{enumerate}[(i)] 
\item The following conditions are equivalent:
\begin{enumerate}[(a)]
\item the homomorphism $\Phi_G$ factors through~$\Phi_{k}: \bsurf{k}{g,n}{}\to G_k\bigl(\surf{g}{}\bigr)$.
\item the restriction of $\Phi_G$ to $B_{k}(\D_n)$ factors through~$p_{k}: B_{k}(\D_n)\to G_k$.
\item there exist $\si,\zeta \in G$ such that $\Phi_{G}(S)= \{\si\}$ and $\Phi_{G}(Z)= \{\z\}$.
\end{enumerate}

\item 
Let $\Phi_{G,2}: G_k\to G$ be a homomorphism such that $\Phi_{G,2}\circ p_{k}$ is the restriction of~$\Phi_{G}$ to $B_{k}(\D_n)$. Then there exists a homomorphism $\Phi_{G,3}: G_k\bigl(\surf{g}{}\bigr)\to G$ such that $\Phi_{G,3}\circ \Phi_{k}=\Phi_{G}$. Further, $\Phi_{G,3}$ is injective if and only if $\Phi_{G,2}$ is injective.
\end{enumerate}
\end{cor}

\begin{proof}\mbox{}
\begin{enumerate}[(i)]
\item Suppose first that $\Phi_G$ factors through~$\Phi_{k}$. By the proof of Theorem~\ref{main_thm}, $\Phi_{k}(S) = \{\si\}$ and $\Phi_{k}(Z) = \{\z\}$, so using the same notation for the corresponding elements of $G$, there exist $\si,\zeta \in G$ such that $\Phi_{G}(S)= \{\si\}$ and $\Phi_{G}(Z)= \{\z\}$, hence~(a) implies~(c). Since $\gamma_k\circ p_k = \Phi_k\circ\iota_k$ by Theorem~\ref{main_thm_2}(i), the restriction of $\Phi_G$ to $B_{k}(\D_n)$ factors through~$p_{k}$, so~(a) implies~(b).

Now suppose that~(c) holds. Then $\Phi_{G}(S)= \{\si\}$, and by Corollary~\ref{themain2iv}(i), $\Phi_G$ factors through $\rho_k$, so there exists a homomorphism $\Phi'_G: \bsurf{k}{g,n}{}/\Gamma_3(\bsurf{k}{g,n}{})\to G$ such that $\Phi_G = \Phi'_G \circ\rho_k$. Further, $\Phi_{G}(Z)= \{\z\}$, and letting $Z$ also denote the usual subset of elements of $\bsurf{k}{g,n}{}$, we have that $\rho_{k}(Z)=Z$, and thus $\Phi'_G(Z)=\{ \z \}$. So $\Phi'_G$ factors through the canonical surjection $\bsurf{k}{g,n}{}/\Gamma_3(\bsurf{k}{g,n}{}) \to M_k\bigl(\surf{g}{}\bigr)$. Taking $q_k: \bsurf{k}{g,n}{}/\Gamma_3(\bsurf{k}{g,n}{}) \to G_k\bigl(\surf{g}{}\bigr)$ to be a homomorphism as in the statement of Corollary~\ref{cor:decgamma4}(iii) that satisfies $\Phi_k = q_k\circ \rho_k$, and applying the resulting induced isomorphism between $G_k\bigl(\surf{g}{}\bigr)$ and $M_k\bigl(\surf{g}{}\bigr)$, we obtain a homomorphism $\Phi_{G,3}: G_k\bigl(\surf{g}{}\bigr)\to G$ such that $\Phi_G' = \Phi_{G,3}\circ q_k$. Hence $\Phi_G = \Phi'_G \circ\rho_k=\Phi_{G,3}\circ q_k\circ\rho_k= \Phi_{G,3}\circ \Phi_{k}$, therefore $\Phi_G$ factors through $\Phi_k$, and thus~(c) implies~(a).


Finally, suppose that~(b) holds. Then $p_{k}(S)=r_{k,n}(S)=\{ \si \}$ and $p_{k}(Z)=r_{k,n}(Z)=\{ \z \}$ by the commutative diagram~(\ref{eq:diagartin}) and Proposition~\ref{lem:presgammaabel}. Since $\Phi_{G}\circ \iota_{k}$ factors through $p_{k}$, it follows that there exist $\si,\zeta \in G$ such that $\Phi_{G}(S)= \{\si\}$ and $\Phi_{G}(Z)= \{\z\}$, hence~(b) implies~(c).

\item Since the restriction of $\Phi_G$ to $B_{k}(\D_n)$ factors through~$p_{k}$, it follows from part~(i) that $\Phi_G$ factors through~$\Phi_{k}$, and so there exists a homomorphism $\Phi_{G,3}: G_k\bigl(\surf{g}{}\bigr)\to G$ such that $\Phi_{G,3}\circ \Phi_{k}=\Phi_{G}$. By hypothesis, we have $\Phi_{G}\circ \iota_{k}=\Phi_{G,2}\circ p_{k}$, so $\Phi_{G,2}\circ p_{k}= \Phi_{G}\circ \iota_{k}= \Phi_{G,3}\circ \Phi_{k}\circ \iota_{k}= \Phi_{G,3}\circ \gamma_{k} \circ p_{k}$. The surjectivity of $p_{k}$ implies that $\Phi_{G,2}= \Phi_{G,3}\circ \gamma_{k}$. Now $\gamma_{k}$ is injective by Theorem~\ref{main_thm_2}(i), and applying Proposition~\ref{fondhk} and Corollary~\ref{cor:decgamma4}(iii), we see that $\Gamma_{2}(G_k\bigl(\surf{g}{}\bigr))\subset Z(G_k\bigl(\surf{g}{}\bigr))$ and $Z(G_k\bigl(\surf{g}{}\bigr))$ is equal to the image of $\gamma_{k}$. The result then follows from Lemma~\ref{lem:equivinjective}.\qedhere
\end{enumerate}   
\end{proof}

This enables us to prove Theorem~\ref{main_thm_2}(ii).
\begin{proof}[Proof of Theorem~\ref{main_thm_2}(ii)]
With the notation of the proof of Corollary~\ref{mainthipr2}(iv), the existence and injectivity of $\Phi_{G,2}$ imply those of $\Phi_{G,3}: G_k\bigl(\surf{g}{}\bigr)\to G$. The surjectivity of $\Phi_{G,3}$ follows from that of $\Phi_{G}$. We thus have $\Phi_{G,3}\circ \Phi_{k}=\Phi_{G}$, where $\Phi_{G,3}$ is an isomorphism.
\end{proof}

Taking into account the proof of parts~(i),~(iii) and~(iv) in Section~\ref{sec:metaquotients2}, this concludes the proof of Theorem~\ref{main_thm_2}. 

\smallskip

These results allow us also to do make some remarks on the extension of the length function from $B_n$ to surface braid groups.
Let $k\geq 3$. The projection $r_{k}: B_k\to B_{k}/\Gamma_2(B_{k})$ coincides (up to isomorphism) with the usual length function $\lambda: B_k\to \Z$ on $B_{k}$. If $g\geq 1$, we claim that, up to isomorphism, the only surjective homomorphism extending $\lambda$ from $B_k$ to $\bsurf{k}{g}{}$ is $\rho_{k}: \bsurf{k}{g}{}\to \bsurf{k}{g}{}/\Gamma_3(\bsurf{k}{g}{})$. Indeed, let $G$ be a group, and let $\lambda_{\Sigma}: \bsurf{k}{g}{}\to G$ be a surjective homomorphism that extends $\lambda$. So there exists an injective homomorphism $\rho_{G,2}: B_{k}/\Gamma_2(B_{k})\to G$ satisfying $\rho_{G,2}\circ \lambda=\lambda_{\Sigma} \circ \iota_{k}$. Applying Corollary~\ref{mainthipr2}(ii), there exists an isomorphism $\rho_{G,3}: \bsurf{k}{g}{}/\Gamma_3(\bsurf{k}{g}{})\to G$ such that $\rho_{G,3}\circ \rho_{k}=\lambda_{\Sigma}$, which proves the claim.
In the case of surfaces without boundary, we have the following negative result, which as we mentioned at the beginning of Section~\ref{section3}, provides another reason why we only consider surfaces with boundary in this paper.

\begin{prop} \label{fondpropcons2}
Let $n \ge  3$, and let $\Sigma$ be a compact orientable surface of positive genus without boundary. It is not possible to extend  the length function $\lambda: B_n \to \Z$ to $ B_{n}(\Sigma) $. In other words, there does not exist a surjective homomorphism $\lambda_{\Sigma}$ of $B_{n}(\Sigma) $ onto a group $F$ whose restriction $B_n$ coincides with
$\lambda$. 
\end{prop}

\begin{proof}
Let $F$ be a group, and let $\lambda_{\Sigma}: B_{n}(\Sigma) \to F$ be a homomorphism that extends $\lambda$. Then there exists $\sigma\in F$ such that $\sigma=\lambda_{\Sigma}(\sigma_1)=\cdots= \lambda_{\Sigma}(\sigma_{n-1})$. The group presentation of $ B_{n}(\Sigma)$ given in~\cite{B} implies that $\sigma^{2(n+g-1)}=1$ (for further details, see the proof of~\cite[Theorem~1]{BGG}). The result then follows because $\sigma$ is of finite order but $\lambda(\sigma_{1})$ is of infinite order.
\end{proof}

\smallskip

In~\cite[Section~3]{HK}, the authors consider also a group $H_\Sigma$ that is defined by its group presentation. One may check using this presentation that~$H_\Sigma$ is isomorphic to the quotient $\bigl(B_{k,n}(\surf{g}{})/ \Gamma_3(B_{k,n}(\surf{g}{}))\bigr)/_{\tsi = 1}$. In~\cite[Theorem~4.3(ii)]{HK}, a rigidity result is proved for this group in the case $k\geq 3$. 
 We conclude this section with an alternative proof of this theorem, based mainly on the results and arguments of Section~\ref{section3}.

\begin{prop}\label{prop:final}
Let $k,n\geq 3$. Let $H$ be a group, and let $\Phi_H : B_{k,n}(\surf{g}{})\to H$ be a homomorphism such that:
\begin{enumerate}[(i)]
\item the restriction of $\Phi_H$ to~$B_{k,n}$ factors through $r_{k,n}$, in other words, there exists a homomorphism $\Phi_{H,2}: B_{k,n}/\Gamma_{2}(B_{k,n})\to H$ such that $\Phi_H\circ \iota_{k,n}=\Phi_{H,2}\circ r_{k,n}$.
\item the kernel of~$\Phi_{H,2}$ is generated by $\tsi$.
\end{enumerate}
Then $\Phi_H$ induces an injective homomorphism from~$H_\Sigma$ to $H$. In particular, if $\Phi_H$ is surjective then the induced homomorphism is an isomorphism.        
\end{prop}

Note that the assumptions~(i) and~(ii) of Proposition~\ref{prop:final} are equivalent to  the following two conditions:
\begin{enumerate}[(a)]
\item the restriction of $\Phi_H$ to $B_{k}(\D_n)$ induces an injective homomorphism from $G_k$  to $H$.
\item $\Phi_H(\widetilde{S}) = \{1\}$.
\end{enumerate}

We remark that conditions~(a) and~(b) correspond to Hypothesis~(\dag) in~\cite[Theorem~4.3(ii)]{HK}.

\begin{proof}
Since the restriction of $\Phi_H$ to~$B_{k,n}$ factors through $r_{k,n}$ using hypothesis~(i), Corollary~\ref{cor:themain2iii}(i) implies that $\Phi_H$ factors through~$\rho_{k,n}$, so there exists a homomorphism $\Phi_{H,3}: B_{k,n}(\surf{g}{})/ \Gamma_3(B_{k,n}(\surf{g}{}))\to H$ such that $\Phi_H = \Phi_{H,3}\circ \rho_{k,n}$. Now $\Phi_H(\widetilde{S}) = \{1\}$ using hypothesis~(ii) and $\rho_{k,n}(\widetilde{S})=\brak{\tsi}$ by Proposition~\ref{lem:presgamma}, so it follows that $\Phi_{H,3}(\tsi) = \{1\}$. Thus the homomorphism~$\Phi_{H,3}$ factors through the projection of the group~$B_{k,n}(\surf{g}{})/\Gamma_3(B_{k,n}(\surf{g}{}))$ onto its quotient~$\bigl(B_{k,n}(\surf{g}{})/ \Gamma_3(B_{k,n}(\surf{g}{}))\bigr)/_{\tsi = 1}$, which from the preceding remarks, we know to be isomorphic to $H_\Sigma$. It remains to prove that the induced homomorphism from this quotient (or equivalently from $H_{\Sigma}$) to $H$ is injective. To do so, first note that $\Phi_{H,3}\circ \alpha_{k,n}= \Phi_{H,2}$, where $\alpha_{k,n}:B_{k,n}/ \Gamma_2(B_{k,n})\to B_{k,n}(\surf{g}{})/ \Gamma_3(B_{k,n}(\surf{g}{}))$ is the homomorphism given by Corollary~\ref{cor:themain2iii}(ii). Since $\alpha_{k,n}\circ r_{k,n}=\rho_{k,n}\circ \iota_{k,n}$, we have $\alpha_{k,n}(\tsi)=\tsi$, and so the subgroup of~$B_{k,n}(\surf{g}{})/\Gamma_3(B_{k,n}(\surf{g}{}))$ generated by~$\tsi$ is contained in $\ker{(\Phi_{H,3})}$. The proof of the converse is similar in sprit to that of Lemma~\ref{lem:equivinjective}. Let $\gamma \in \ker{(\Phi_{H,3})}$, and let $\gamma'\in B_{k,n}(\surf{g}{})/\Gamma_3(B_{k,n}(\surf{g}{}))$. The element~$[\gamma,\gamma']$ belongs to the centre of $ B_{k,n}(\surf{g}{})/\Gamma_3(B_{k,n}(\surf{g}{}))$, and therefore to the image of $\alpha_{k,n}$ by Corollary~\ref{cor:themain2iii}(ii).  Let $h\in B_{k,n}/\Gamma_2(B_{k,n})$ be such that $\alpha_{k,n}(h) = [\gamma,\gamma']$. Then~$\Phi_{H,2}(h) =  \Phi_{H,3}([\gamma,\gamma']) = [\Phi_{H,3}(\gamma),\Phi_{H,3}(\gamma')] = 1$. Therefore $h = \tsi^\ell$ for some $\ell\in \mathbb{Z}$ using hypothesis~(ii), and the relation $\alpha_{k,n}(\tsi)=\tsi$ implies that for all $\gamma'\in B_{k,n}(\surf{g}{})/ \Gamma_3(B_{k,n}(\surf{g}{}))$, the commutator~$[\gamma,\gamma']$ is equal to a power of $\tsi$. Write $\gamma$ in the form of equation~(\ref{eq:gamma}). Then for all $j=1,\ldots,g$, we have $[\gamma,b_j] = (\sigma^{2m_j})\z^{\widetilde{m}_j}$ and $[\gamma,a_j] = (\si^{-2n_j})\z^{-\widetilde{n}_j}$ by Proposition~\ref{lem:presgamma}. Since these commutators are powers of $\tsi$, we deduce that $m_j = n_j = \widetilde{m}_j = \widetilde{n}_j=0$, and so $\gamma$ belongs to the centre of~$B_{k,n}(\surf{g}{})/\Gamma_3(B_{k,n}(\surf{g}{}))$ by Corollary~\ref{cor:decgammamix}(iii). Repeating the argument with $\gamma$ in place of $[\gamma,\gamma']$, there exists $h'\in B_{k,n}/ \Gamma_2(B_{k,n})$ such that $\alpha_{k,n}(h') = \gamma$, so $\Phi_{H,2}(h') = 1$ because $\Phi_{H,3}\circ \alpha_{k,n}= \Phi_{H,2}$, and thus $\gamma = \tsi^\ell$ for some $\ell\in \Z$ by hypothesis~(ii) and the fact that $\alpha_{k,n}(\tsi)=\tsi$. The kernel of  $\Phi_{H,3}$ is therefore the subgroup of~$B_{k,n}(\surf{g}{})/\Gamma_3(B_{k,n}(\surf{g}{}))$ generated by~$\tsi$, and hence the induced homomorphism from $H_\Sigma$ to $H$ is injective. 
\end{proof}


\section{Representations of surface braid groups}\label{section:representations}


In this section, we first describe an algebraic approach to the Burau and Bigelow-Krammer-Lawrence representations that is based on the lower central series. To our knowledge, this description has not appeared in the literature, although it is well known to the experts in the field. We then show why it is not possible
to extend the Burau and Bigelow-Krammer-Lawrence representations to representations of surface braid groups. This latter remark was made in~\cite{HK} under certain conditions of a homological nature. Within a purely algebraic framework, we prove this non-existence result with fewer conditions than those of~\cite{HK}. 

Let us start recalling that $B_{n}$ may be seen as the mapping class group of $\D_n$, thus giving rise to an action of $B_n$ on $\pi_1(\D_n)$, the latter being isomorphic to the free group $F_n$ on $n$ generators~\cite{Bir}. This action coincides with the action by conjugation of $B_n$ on $B_1(\D_n)$ defined by the standard section of~(\ref{eq:sequence}), and gives rise to Artin's (faithful) representation of $B_n$ as a subgroup of the group of automorphisms of $F_n$. The (non-reduced) Burau representation of $B_n$ is then obtained by composing the Artin representation with the Magnus representation associated with the length function~$\ell: B_1(\D_n)\to \mathbb{Z}$ (see for instance~\cite{Bar}), which we identify with the homomorphism $p_1: B_1(\D_n)\to G_1
$ of the commutative diagram~(\ref{eq:diagartin}) (see Lemma~\ref{lemGkpres} for more details).

This representation also has a homological interpretation (see for instance~\cite[Chapter 3]{KT}): the group $B_n$ acts on the infinite cyclic covering $\widetilde{\D}_n$ of $\D_n$, and since the action of $B_n$ on $B_1(\D_n)$ commutes with the length function $p_1$ (whose image is $\Z$), the induced action on the first homology group of $\widetilde{\D}_n$ is the (reduced) Burau representation of $B_n$. 

More generally, if $k \geq 1$, $B_n$, regarded as the mapping class group  of $\D_n$, acts on $\FF_k(\D_n)/S_{k}$ and therefore on its fundamental group $B_k(\D_n)$. Once more, the induced action of $B_n$ on $B_k(\D_n)$ coincides with the action by conjugation of $B_n$ on $B_k(\D_n)$ defined by the standard section of~(\ref{eq:sequence}). In analogy with the case $k=1$, in order to seek (linear) representations of $B_k(\D_n)$, we consider regular coverings associated with its normal subgroups, and we try to see if the induced action on homology is well defined. In other words, we wish to study surjections of $B_k(\D_n)$ onto a group subject to certain constraints. Now suppose that $k>1$, and that the group~$G_k$ is a free Abelian group of rank~$2$ ({\it cf.} Lemma~\ref{lemGkpres}), and consider the homomorphism $p_k$ of the commutative diagram~(\ref{eq:diagartin}). It is easy to check that the action of $B_n$ on $B_k(\D_n)$ commutes with $p_k$, that $B_n$ acts on the regular covering of $\FF_k(\D_n)/S_k$ corresponding to $p_{k}$, and that the induced action on the Borel-Moore middle homology group of this covering space
defines 
a homological representation of $B_n$. When $k=2$, we obtain the famous Bigelow-Krammer-Lawrence representation that is a faithful linear representation (see~\cite{KT} for a complete description of this construction). In~\cite{Z}, it was proved that the corresponding representations are also faithful in the general case $k\geq 2$. In what follows, we will call this family of representations \emph{BKL representations} (the Burau and the usual Bigelow-Krammer-Lawrence representations correspond to the cases $k=1$ and $k=2$). 

With this algebraic constructions of Section~3 in mind, we would like to extend the BKL representations from $B_n$ to $B_n(\surf{g}{})$. To do so, one might seek a projection $\Phi_G: B_k(\surf{g}{, n}) \to G$ onto an Abelian group $G$ that induces an action on the homology of the corresponding covering and whose restriction to $B_{k}(\D_n) \subset B_k(\surf{g}{, n})$ coincides with 
$p_k:  B_k(\D_n) \to G_k$ (more precisely, there is an injective homomorphism~$j_k$ from~$G_k$ to $G$ so that the restriction of $\Phi_G$ to $B_{k}(\D_n)$ coincides with~$j_k\circ p_k$). However, if there existed $\Phi_G$ such that $\Phi_G \circ \beta_* =\Phi_G$ for all $\beta \in  B_n(\surf{g}{})$, where $\beta_*$ denotes the automorphism of $B_k(\surf{g}{, n})$ induced by conjugation by~$\beta$, then the homomorphism $\Phi_G$ would extend to a homomorphism from~$B_{k,n}(\surf{g}{})$ to~$G$ because $G$ is Abelian, and we would obtain a linear representation of $B_n(\surf{g}{})$ in $\operatorname{Aut}_{\Z[G]}\Bigl(H_k^{\text{BM}} \Bigl( \widetilde{E_{k}/{S_k}} \Bigr)\Bigr)$, where  $H_k^{\text{BM}} \left( \widetilde{E_{k}/{S_k}} \right)$ is the Borel-Moore middle homology group of the covering space
of $E_{k}/{S_k}$ and $E_{k}=\FF_k(\surf{g}{, n})$. Unfortunately, the results of Section~\ref{section3} imply that this approach is not valid when $k\ge 2$, even if we do not impose the equalities~$\Phi_G \circ \beta_* =\Phi_G$.

\begin{lem} \label{nolinearext}
Let $k\geq 1$ and $n\geq 1$, and let $G$ be an Abelian group.
\begin{enumerate}[(i)]
\item Let $k\geq 2$. Suppose that there exist homomorphisms $\Phi_G: B_k(\surf{g,n}{}) \to G$  and $j_k: G_k \to G$ that satisfy~$\Phi_G\circ \iota_{k}= j_k \circ p_k$. Then $j_k$ is not injective.

\item Let $k = 1$. For every non-trivial homomorphism $\Phi_G: B_1(\surf{g,n}{}) \to G$, there exists an injective homomorphism~$j_1: G_1 \to G$ that satisfies~$\Phi_G \circ j_{1}= j_1 \circ p_1$. 
\end{enumerate}
\end{lem}

\begin{proof}\mbox{}
\begin{enumerate}[(i)]
\item  Let $\Phi_G$ and $j_k$ be as in the statement. Since $G$ is Abelian, the homomorphism $\Phi_G$ factors through $\widehat{r}_k: B_k(\surf{g,n}{})\to B_k(\surf{g,n}{})/\Gamma_2(B_k(\surf{g,n}{}))$. By Proposition~\ref{lem:presbngabel}, $(\widehat{r}_k(\si_1))^2 = \si^2 = 1$, and so $\Phi_G(\si_1^2)=1$. On the other hand, $p_{k}(\si_{1})=\si$, and this element is of infinite order in $G_{k}$. The relation $\Phi_G \circ j_{k}= j_k \circ p_k$ implies the non-injectivity of $j_k$.

\item Assume $k = 1$. As in~(i), the homomorphism~$\Phi_G$ factors through $\widehat{r}_{1}$. But $G_1$ is isomorphic to $\mathbb{Z}$ and generated by $p_1(\z_1)$ by Lemma~\ref{lemGkpres}; on the other hand $\widehat{r}_1(\z_1)$ is torsion free by Proposition~\ref{lem:presbngabel}. Thus $\widehat{r}_1\circ j_1$ factors through $p_1$.\qedhere
\end{enumerate}
\end{proof}

This means that we cannot construct a linear representation for surface braid groups whose restriction to $B_{n}$ is the Bigelow-Krammer-Lawrence representation. Furthermore if we impose that~$\Phi_G \circ \beta_* =\Phi_G$ for any $\beta \in  B_n(\surf{g}{})$, this negative result can be extended to the case $k=1$ and to any group $G$. The following proposition is a reformulation in our framework of~\cite[Lemma~2.6]{HK}  and of related  remarks on the homological constraints.

\begin{prop}\label{nolinearext2}
Let $k\geq 1$ and $n\geq 2$, and let $G$ be a group. Suppose that there exist homomorphisms $\Phi_G: B_k(\surf{g,n}{}) \to G$  and $j_k: G_k \to G$ that satisfy~$\Phi_G\circ \iota_{k}= j_k \circ p_k$ and~$\Phi_G \circ \beta_* =\Phi_G$ for any $\beta \in  B_n(\surf{g}{})$.
Then $j_k$ is not injective.
\end{prop}

\begin{proof}
The action of $B_n(\surf{g}{})$ on $B_k(\surf{g,n}{})$ is completely described in Proposition~\ref{presconnues}.
As remarked in~\cite[Lemma~2.6]{HK}, combining relation~(c.7.1) (we need therefore to suppose $n\neq 1$) with the hypothesis ~$\Phi_G \circ \beta_* =\Phi_G$, we see that~$\Phi_G(\zeta_1)=1$. On the other hand $p_{k}(\zeta_{1})=\zeta$ by Proposition~\ref{lem:presbngabel}. The relation $\Phi_G \circ j_{k}= j_k \circ p_k$ implies the non-injectivity of $j_k$.
\end{proof}




 



\noindent
{\bf Acknowledgements.}
The research of the first author was supported  by  the French grant ANR-11-JS01-002-01, and that of the second and third authors was supported by the French grant ANR-08-BLAN-0269-02.

\vspace{10pt}

\noindent Paolo BELLINGERI,  Eddy GODELLE and  John GUASCHI,

\noindent Universit\'e de Caen Basse-Normandie, Laboratoire de Math\'ematiques Nicolas Oresme CNRS UMR 6139, BP 5186,  F-14032 Caen (France).

{\small
\noindent \textit{Email}: \url{paolo.bellingeri@unicaen.fr}, \url{eddy.godelle@unicaen.fr}, \url{john.guaschi@unicaen.fr}}

\end{document}